\RequirePackage[final]{graphicx}
\documentclass[final]{siamltex}

\usepackage{tikz}
\usetikzlibrary{calc}
\usepackage{pgfplots}
\usepackage{bm}
\usepackage{amsmath,amsfonts}
\renewcommand{\vec}[1]{{\boldsymbol{#1}}}
\usepackage{booktabs}
\usepackage{type1cm}
\usepackage{color}
\usepackage{url}
\newcommand{\mat}[1]{{\boldsymbol{#1}}}
\newcommand{\pd}[2]{\ensuremath{\frac{\partial #1}{\partial #2}}}
\newcommand{\xix}[2]{\ensuremath{\frac{\mat{\partial \xi}_{#1}}{\mat{\partial x}_{#2}}}}
\newcommand{\fd}[2]{\frac{d #1}{d #2}}
\newcommand{\tref}[1]{Table~\ref{#1}}
\newcommand{\fref}[1]{Fig.~\ref{#1}}
\newcommand{\eref}[1]{(\ref{#1})}

\newcommand{\sref}[1]{\S~\ref{#1}}

\newcommand{\aref}[1]{Appendix~\ref{#1}}
\newcommand{\kron}{\otimes}
\newcommand{\ointerval}[2]{(#1, #2)}
\newcommand{\cinterval}[2]{[#1, #2]}
\definecolor{red}{rgb}{0.7592, 0.3137, 0.3020}
\definecolor{blu}{rgb}{0.3098, 0.5059, 0.7412}
\definecolor{grn}{rgb}{0,0.4,0}
\definecolor{org}{rgb}{1.0,0.5,0.0}
\definecolor{prp}{rgb}{0.4,0.35,0.8}

\title{Stable Coupling of Nonconforming, High-Order Finite Difference Methods}

\author{Jeremy E.~Kozdon\footnotemark[2] \and
Lucas C.~Wilcox\footnotemark[2]}

\begin{document}

\maketitle
\renewcommand{\thefootnote}{\fnsymbol{footnote}}

\footnotetext[2]{Department of Applied Mathematics, Naval Postgraduate School,
   Monterey, CA, 93943--5216 (\texttt{\{jekozdon,lwilcox\}@nps.edu})}

\renewcommand{\thefootnote}{\arabic{footnote}}

\begin{abstract}
  A methodology for handling block-to-block coupling of nonconforming,
  multiblock summation-by-parts finite difference methods is proposed.
  The coupling is based on the construction of projection operators that move a
  finite difference grid solution along an interface to a space of piecewise
  defined functions; we specifically consider discontinuous, piecewise
  polynomial functions.
  The constructed projection operators are compatible with the underlying
  summation-by-parts energy norm.
  Using the linear wave equation in two dimensions as a model problem,
  energy stability of the coupled numerical method is proven for the case of
  curved, nonconforming block-to-block interfaces.
  To further demonstrate the power of the coupling procedure, we show how it
  allows for the development of a provably energy stable coupling between
  curvilinear finite difference methods and a curved-triangle discontinuous
  Galerkin method.
  The theoretical results are verified through numerical simulations on
  curved meshes as well as eigenvalue analysis.
\end{abstract}



\pagestyle{myheadings}
\thispagestyle{plain}
\markboth{J. E. KOZDON AND L. C. WILCOX}{NONCONFORMING FINITE DIFFERENCE}

\section{Introduction}
Even though high-order multiblock finite difference methods are well suited for
many problems, limitations arise for particularly complex geometries.
For instance, most formulations require that grids conform at multiblock
interfaces, that is the grids lines must be continuous. This poses a challenge
since resolution constraints in one portion of the domain can result in
unnecessarily high resolution elsewhere in the domain.
Furthermore, even though coordinate transforms enable the use of high-order
finite difference methods for complex geometries, well-conditioned partitioning
of complex domains into quadrilaterals in two-dimensions and hexahedra
in three-dimensions can be challenging and/or impossible. The impact of this is
a poorly conditioned Jacobian which has an adverse impact on the time step size
and truncation error.

One approach to overcome these complications is to relax the requirement that
the grid and numerical methods conform across block interfaces. To do this a
variety of interpolation and projection techniques have been proposed including
the use of overlapping grids~\cite{ChesshireHenshaw1990}, strong enforcement
using ghost points~\cite{PeterssonSjogreen2010}, or weak-enforcement of
continuity at block boundaries~\cite{MattssonCarpenter2010}. Here we
particularly highlight the approach of Mattsson and
Carpenter~\cite{MattssonCarpenter2010} as it is closely related to the work
presented below on the discretization of hyperbolic equations using
summation-by-parts (SBP) finite difference methods.
In that paper, compatibility conditions between interpolation operators
and the underlying SBP finite difference method were presented which could be
utilized to developed stable discretizations. The paper reported
several compatible interpolation operators for fixed refinement ratio
interfaces.
The operators proposed by Mattsson and Carpenter required that at the block
level the interfaces be conforming (i.e., the corners had to match), a restriction
which has been removed by the work of Nissen, Kormann, Grandin, and
Virta~\cite{NissenKormannGrandinVirta2014} (this later work retains the fixed
refinement ratio requirement).
To avoid geometric constraints, it has also been proposed to couple high-order
finite difference methods with unstructured grid methods.
For example, Nordstr\"{o}m and Gong~\cite{NordstromGong2006} proposed coupling a
high-order SBP method with an unstructured second-order finite volume method.

The class of finite difference methods to be considered in this work are SBP
finite difference methods~\cite{KreissScherer1974, KreissScherer1977,
Strand1994, MattssonNordstrom2004}; see \sref{sec:definitions} for the basic SBP
ideas used in this work.
One important feature of SBP methods is that the difference operators
has an associated energy norm that discretely mimics integration by parts;
this is referred to as the SBP property.

\begin{figure}[tb]
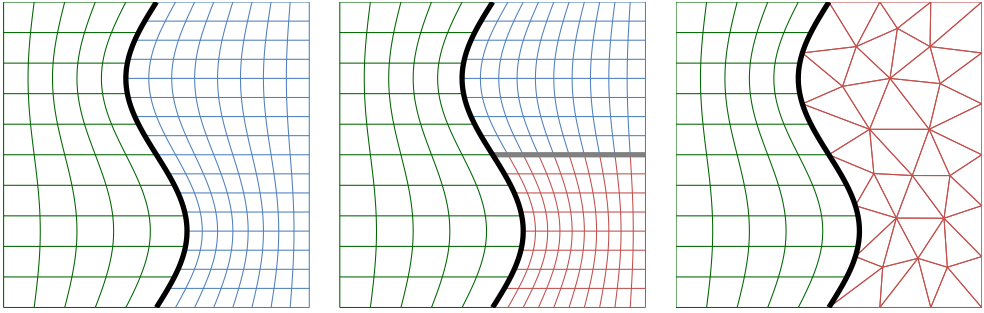

  \centering
  \input{sbp_2block.tex}
  \hfill
  \input{sbp_3block.tex}
  \hfill
  \input{sbpdg.tex}
  \caption{Illustration of the computational grids supported by the projection
    operators described in this paper. On the left is a conforming two-block
    SBP grid. In the middle is a nonconforming (T-intersection) of three SBP
    blocks; here nonconforming refers to the SBP blocks and not the SBP grids.
    On the right is a coupled structured-unstructured grid; in this paper the
    methods used will be SBP finite differences and DG grid.  In each
    illustration the thick line between blocks indicates coupling
    interfaces, i.e., the locations of the glue grids.}\label{fig:illustration}
\end{figure}
Here we present a general purpose technique for handling block-to-block coupling
of nonconforming multiblock SBP finite difference methods.  This technique also
allows for the coupling with unstructured methods such as the discontinuous
Galerkin (DG) method.  A few possible grid couplings illustrated in
\fref{fig:illustration}.
This coupling procedure uses projection operators that move finite difference
grid solutions along the coupling interface to piecewise functions.
It is with respect to the built-in norm of the SBP finite difference method
that the projection operators are constructed; projection operators with the
necessary properties are given in the electronic supplement. The SBP property
alone is not sufficient to guarantee stability, SBP preserving boundary and
interface closures are also required. In this work we will achieve this by
enforcing all boundary and interface conditions weakly through the so-called
simultaneous approximation term (SAT)
method~\cite{CarpenterGottliebAbarbanel1994}; this is similar to the use of
numerical flux terms in the DG method, a fact which will be exploited to stably
couple SBP and DG methods.

Since the projection operators move the solution to a piecewise continuous
representation (where projections are straight forward to construct), the
operators only need to be constructed once for each finite difference operator.
That is projection operators can be constructed independent of the numerical
method and grid on the other side of an interface. This independence of
interface type is one of the features that enables both development of provably
stable couplings between conforming and nonconforming SBP meshes as well as
between SBP and DG methods.  The fact that nonconforming grids can be
accommodated enables the development of adaptive mesh refinement codes using
high-order, SBP finite difference methods.

For simplicity of presentation, we take as our model problem the two-dimensional
linear acoustic wave equation in first order form.  We prove that the proposed
coupling is stable for this system of equations as well as provide numerical
evidence to confirm the analytical results.  Since the projection operators are
constructed based on the SBP operator, and do not depend on the system of
equations being solved, using standard techniques the extension to other linear
symmetric hyperbolic systems should be possible.

\section{Definitions}\label{sec:definitions}
We begin by stating a few preliminary definitions that are at the heart of this
work.

\begin{definition}[SBP property]
  A difference approximation $\mat{D}$ is called a summation-by-parts
  (SBP) approximation to $d/d x$ if it can be decomposed as
  $\mat{D} = \mat{H}^{-1}\mat{Q}$ with $\mat{H}$ being positive definite
  and $\mat{Q}$ having the property $\mat{Q} + \mat{Q}^{T} = \mat{B} =
  \diag \begin{bmatrix} -1 & 0 & \cdots & 0 & 1 \end{bmatrix}$, i.e., $\mat{Q}$
  is almost skew-symmetric.
\end{definition}

To understand why such a difference approximation is called SBP consider a
grid function $\vec{f} = {\begin{bmatrix} f_{0}&f_{1}&\cdots&f_{N}\end{bmatrix}}^{T}$. The
$\mat{H}$-weighted inner product of $\vec{f}$ and $\mat{D}\vec{f}$ gives
\begin{align}
  {(\vec{f},\mat{D}\vec{f})}_{H}
  = \vec{f}^{T}\mat{H}\mat{D}\vec{f}
  = \vec{f}^{T}\mat{Q}\vec{f}
  = \frac{1}{2}\vec{f}^{T}\left(\mat{Q}+\mat{Q}^{T}\right)\vec{f}
  = \frac{1}{2}\left(f_{N}^{2} - f_{0}^{2}\right),
\end{align}
which is of the same form as the inner product of a continuously differentiable
function $f\in C^1\cinterval{x_l}{x_r}$ and $d f / d x$:
\begin{align}
  {\left(f,\frac{d f}{d x}\right)}_{L^{2}\ointerval{x_l}{x_r}}
  = \int_{x_{l}}^{x_{r}} f\frac{d f}{d x}\;dx
  = \frac{1}{2}\left( f_{r}^{2} - f_{l}^{2} \right),
\end{align}
where $\ointerval{x_l}{x_r}$ is an open interval of the real line $\mathbb{R}$,
$f_l = f(x_l)$, and $f_r = f(x_r)$.

The difference operators commonly referred to as SBP methods are central
difference operators in the interior (with orders: $2$, $4$, $6$, $8$,
\ldots) which transition to one-sided approximations near the boundary in such a
way that the SBP property is achieved~\cite{KreissScherer1974,
  KreissScherer1977, MattssonNordstrom2004, Strand1994}. This transition to
one-sided typically leads to a degradation in accuracy at the boundary. These
SBP operators are subdivided into two classes: diagonal norm (diagonal
$\mat{H}$) and block norm (non-diagonal $\mat{H}$) operators. For the diagonal
norm operators the boundary accuracy can be at most half the interior accuracy,
i.e., if the approximation is $2q$-accurate in the interior it is at most
$q$-accurate at the boundary. For the block norm operators it is possible to
construct difference approximations that are $2q-1$ accurate at the boundary. In
both cases the global accuracy of the scheme is one more than the boundary
accuracy, i.e., $q+1$ for the diagonal norm operators and $2q$ for the block
norm operators~\cite{Gustafsson1975}. For most practical calculations the
diagonal norm operators are used as they result in stable schemes for problems
coordinate transforms and variable
coefficients~\cite{KozdonDunhamNordstrom2012, KozdonDunhamNordstrom2013,
Nordstrom2006, NordstromCarpenter2001, Olsson1995}; a notable exception is the
recent work of Mattsson and Almquist~\cite{MattssonAlmquist2013} where
artificial dissipation is used to stabilize the block norm operators in complex
geometries.

A key concept for this work is the definition of an \emph{SBP
$\mat{H}$-compatible projection operator}. This operator will allow us
to move from a grid function to a space of piecewise continuous
functions in a manner that is compatible (in an $L^{2}(\Gamma)$ sense) with
the SBP finite difference method. We call the space of piecewise
continuous functions the \emph{glue grid} since it allows us to
``glue'' together differing numerical methods.

To make this more concrete, given a finite difference grid
$\begin{bmatrix}x_0& x_1& \cdots& x_N\end{bmatrix}$ let
$\mathcal{G}_{h}\subset L^{2}(\Gamma)$ be a
finite-dimensional space of functions, i.e., the space of functions the glue
grid can represent.
Let $\vec{\psi}(\eta)= {\begin{bmatrix}\psi_0(\eta)& \psi_1(\eta)& \cdots&
\psi_K(\eta)\end{bmatrix}}^T$ be a vector of linearly independent basis functions for
$\mathcal{G}_{h}$, $\vec{f} =
{\begin{bmatrix}f_{0}&f_{1}&\cdots&f_{N}\end{bmatrix}}^{T}$ be a grid
function, and $\mat{H}$ be an SBP norm. Our goal is to define a projection
operator so that a set of coefficients
$\vec{\bar{f}} = {\begin{bmatrix}\bar{f}_{0}& \bar{f}_{1}& \cdots&
    \bar{f}_{K}\end{bmatrix}}^{T}$
can be defined from $\vec{f}$ such that $\bar{f}(\eta) = \sum_{i} \bar{f}_{i}
\psi_{i}(\eta) = \vec{\bar{f}}^{T} \vec{\psi}(\eta)$ is a compatible
representation of the grid function $\vec{f}$ in the space $\mathcal{G}_{h}$.
Note that throughout the paper we use the overline notation to represent
quantities defined on the glue grid.

In order to define the $\mat{H}$-compatible projection operators  we must define
the mass matrix on the glue grid. Namely, the symmetric, positive definite mass
matrix is $\mat{M} = \int_\Gamma \vec{\psi}(\eta)\vec{\psi}^{T}(\eta)\;d\eta$;
thus given two functions $\bar{f}(\eta) = \vec{\bar{f}}^{T}\vec{\psi}(\eta)$
and $\bar{g}(\eta) = \vec{\bar{g}}^{T}\vec{\psi}(\eta)$ in $\mathcal{G}_{h}$
the inner product is
$(\bar{f},\bar{g}) = \vec{\bar{f}}^{T} \mat{M}\vec{\bar{g}}$.

\begin{definition}[$\mat{H}$-Compatible Projection Operator]\label{def:SBPprojection}
  Let $\vec{f}$ be a grid function and $\bar{u}(\eta) =
  \vec{\bar{u}}^{T}\vec{\psi}(\eta)\in \mathcal{G}_h$ be a glue grid function. We call the
  projection matrices $\mat{P}_{f2g}$ and $\mat{P}_{g2f}$ $\mat{H}$-compatible
  if for all $\vec{f}$ and $\vec{\bar{u}}$:
  \begin{align}
    \vec{u}^{T}\mat{H}\vec{f} = \vec{\bar{u}}^{T}\mat{M}\vec{\bar{f}},
  \end{align}
  where $\vec{\bar{f}} = \mat{P}_{f2g} \vec{f}$ and $\vec{u} =
  \mat{P}_{g2f}\vec{\bar{u}}$, or equivalently
  \begin{align}\label{eqn:h:compatible:projections}
    \mat{P}_{g2f}^{T}\mat{H} = \mat{M}\mat{P}_{f2g}.
  \end{align}
  Here the subscript $f2g$ stands for projection from the finite difference grid
  to the glue grid and $g2f$ from the glue grid to the finite difference grid.
\end{definition}

Notice, that nothing in the definition implies that these solutions must be
accurate representations of one another and Definition~\ref{def:SBPprojection}
will only be used to guarantee stability. Furthermore, there is no statement
that the functions can be moved between spaces without error, that is we do not
assume that $\mat{P}_{g2f}\mat{P}_{f2g}\vec{v} = \vec{v}$ nor that
$\mat{P}_{f2g}\mat{P}_{g2f}\vec{\bar{u}} = \vec{\bar{u}}$.

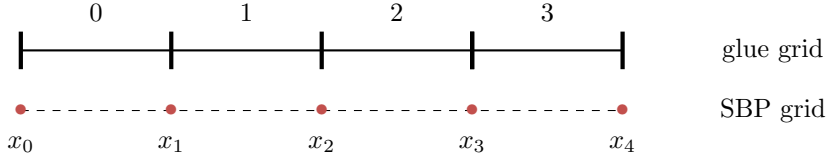
\begin{figure}
  \centering
  \begin{tikzpicture}[scale=2]
    \node at (-2,0.375) {$x_0$};
    \node at (-1,0.375) {$x_1$};
    \node at ( 0,0.375) {$x_2$};
    \node at ( 1,0.375) {$x_3$};
    \node at ( 2,0.375) {$x_4$};

    \draw[dashed] (-2,0.6) to ++(4,0);

    \node at (-2  ,0.6) {\color{red} $\bullet$};
    \node at (-1  ,0.6) {\color{red} $\bullet$};
    \node at ( 0  ,0.6) {\color{red} $\bullet$};
    \node at ( 1  ,0.6) {\color{red} $\bullet$};
    \node at ( 2  ,0.6) {\color{red} $\bullet$};
    \node at ( 3  ,0.6) {SBP grid};

    \draw[thick] (-2,1) to ++(4,0);

    \draw[ultra thick] (-2,1-0.125) to ++(0,0.25);
    \draw[ultra thick] (-1,1-0.125) to ++(0,0.25);
    \draw[ultra thick] ( 0,1-0.125) to ++(0,0.25);
    \draw[ultra thick] ( 1,1-0.125) to ++(0,0.25);
    \draw[ultra thick] ( 1,1-0.125) to ++(0,0.25);
    \draw[ultra thick] ( 2,1-0.125) to ++(0,0.25);
    \node at ( 3  ,1) {glue grid};

    \node at (-1.5,1.25) {$0$};
    \node at (-0.5,1.25) {$1$};
    \node at ( 0.5,1.25) {$2$};
    \node at ( 1.5,1.25) {$3$};

  \end{tikzpicture}
  \caption{Alignment of the glue grid, the line with interval boundaries denoted
    with hatch marks, and the finite difference grid, denoted with dots
    representing the grid points, at the leftmost boundary. The glue grid
    supports functions which are continuous on each interval, e.g., piecewise
    continuous functions. Note that the glue grid and the finite difference
    grid coincide spatially but here we have separated them vertically for
    display purposes.
  }\label{fig:gluegrid}
\end{figure}
It is natural to augment Definition~\ref{def:SBPprojection} with a set of
accuracy conditions based on the particular glue grid space $\mathcal{G}_{h}$
being used. In this work, we let $\mathcal{G}_{h}$ be the space of
discontinuous, piecewise polynomials where the intervals over which the
polynomial are defined align with the finite difference points as shown in
\fref{fig:gluegrid}.  Motivated by Mattsson and
Carpenter~\cite{MattssonCarpenter2010}, we require that the operators used in
this work satisfy a set of polynomial accuracy conditions.
Namely with a glue grid that can represent $q$th order polynomials exactly, we
define the $q$th order polynomial grid
function as $\vec{f}_{q} = {\begin{bmatrix}0^q&1^q&\cdots&N^q\end{bmatrix}}^{T}$
(with the convention that $0^0 = 1$) and let $\bar{g}_{q}(\eta) = \bar{\vec{g}}_{q}^{T}
\vec{\psi}(\eta)$ be the same polynomial on the glue grid. We then require
that the errors
\begin{align}
  \label{eqn:accuracy}
  \vec{e}_{g2f} = \mat{P}_{g2f} \vec{\bar{g}}_{q} - \vec{f}_{q},\qquad
  \vec{e}_{f2g} = \mat{P}_{f2g} \vec{f}_{q} - \vec{\bar{g}}_{q}
\end{align}
be zero for all polynomials up to order $q_{i}-1$ everywhere except near the
boundary where it is required the error be zero for polynomials up to order
$q_{b}-1$; here $q_{i}$ and $q_{b}$ are the interior and boundary accuracy of
the SBP finite difference method being used. In other words, we require the
projection operators to mimic the accuracy of the SBP finite difference method.
These are the same accuracy conditions used in the finite difference to finite
difference operators of Mattsson and Carpenter~\cite{MattssonCarpenter2010}.
The method we use to construct projection operators that satisfy the above
accuracy and stability conditions using a space of discontinuous, piecewise
polynomials are discussed in \aref{app:projection}. Additionally, the electronic
supplement to this paper contains code to generate the operators as well as the
operators themselves.

We emphasize again that the only requirement for a projection operator to result
in a stable discretization is Definition~\ref{def:SBPprojection}. The accuracy
conditions \eref{eqn:accuracy} only pertain to the specific choice of
$\mathcal{G}_{h}$ in this paper and other conditions may be required for
different glue grid spaces.

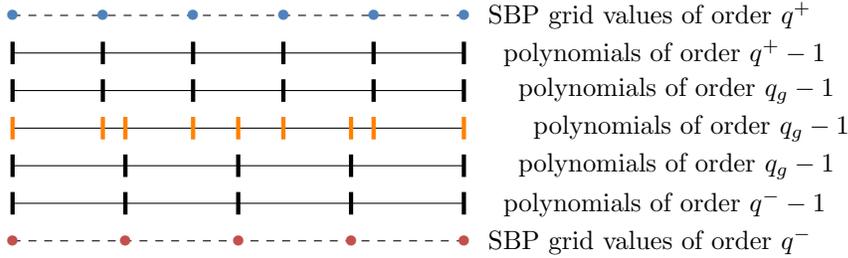
\begin{figure}
  \centering
  \begin{tikzpicture}[scale=2]

    \node[anchor=west] at ( 1.1  ,0.5) {SBP grid values of order $q^{-}$};
    \draw[dashed] (-2,0.5) to ++(3,0);

    \node at (-2.00  ,0.5) {\color{red} $\bullet$};
    \node at (-1.25  ,0.5) {\color{red} $\bullet$};
    \node at (-0.50  ,0.5) {\color{red} $\bullet$};
    \node at ( 0.25  ,0.5) {\color{red} $\bullet$};
    \node at ( 1.00  ,0.5) {\color{red} $\bullet$};

    \node[anchor=west] at ( 1.2  ,0.75) {polynomials of order $q^{-}-1$};
    \draw                  (-2,0.75) to ++(3,0);
    \draw[ultra thick,color=black ] (-2.00  ,0.75-0.075) to ++(0,0.15);
    \draw[ultra thick,color=black ] (-1.25  ,0.75-0.075) to ++(0,0.15);
    \draw[ultra thick,color=black ] (-0.50  ,0.75-0.075) to ++(0,0.15);
    \draw[ultra thick,color=black ] ( 0.25  ,0.75-0.075) to ++(0,0.15);
    \draw[ultra thick,color=black ] ( 1.00  ,0.75-0.075) to ++(0,0.15);

    \node[anchor=west] at ( 1.3,1) {polynomials of order $q_{g}-1$};
    \draw                  (-2,1) to ++(3,0);
    \draw[ultra thick,color=black] (-2.00  ,1.00-0.075) to ++(0,0.15);
    \draw[ultra thick,color=black] (-1.25  ,1.00-0.075) to ++(0,0.15);
    \draw[ultra thick,color=black] (-0.50  ,1.00-0.075) to ++(0,0.15);
    \draw[ultra thick,color=black] ( 0.25  ,1.00-0.075) to ++(0,0.15);
    \draw[ultra thick,color=black] ( 1.00  ,1.00-0.075) to ++(0,0.15);

    \node[anchor=west] at ( 1.4  ,1.25) {polynomials of order $q_{g}-1$};
    \draw        (-2,1.25) to ++(3,0);
    \draw[ultra thick,color=org] (-2.00  ,1.25-0.075) to ++(0,0.15);
    \draw[ultra thick,color=org] (-1.4   ,1.25-0.075) to ++(0,0.15);
    \draw[ultra thick,color=org] (-1.25  ,1.25-0.075) to ++(0,0.15);
    \draw[ultra thick,color=org] (-0.8   ,1.25-0.075) to ++(0,0.15);
    \draw[ultra thick,color=org] (-0.50  ,1.25-0.075) to ++(0,0.15);
    \draw[ultra thick,color=org] (-0.2   ,1.25-0.075) to ++(0,0.15);
    \draw[ultra thick,color=org] ( 0.25  ,1.25-0.075) to ++(0,0.15);
    \draw[ultra thick,color=org] ( 0.4   ,1.25-0.075) to ++(0,0.15);
    \draw[ultra thick,color=org] ( 1.00  ,1.25-0.075) to ++(0,0.15);

    \node[anchor=west] at ( 1.3  ,1.50) {polynomials of order $q_{g}-1$};
    \draw                   (-2,1.50) to ++(3,0);
    \draw[ultra thick,color=black] (-2     ,1.50-0.075) to ++(0,0.15);
    \draw[ultra thick,color=black] (-1.4   ,1.50-0.075) to ++(0,0.15);
    \draw[ultra thick,color=black] (-0.8   ,1.50-0.075) to ++(0,0.15);
    \draw[ultra thick,color=black] (-0.2   ,1.50-0.075) to ++(0,0.15);
    \draw[ultra thick,color=black] ( 0.4   ,1.50-0.075) to ++(0,0.15);
    \draw[ultra thick,color=black] ( 1     ,1.50-0.075) to ++(0,0.15);

    \node[anchor=west] at ( 1.2,1.75) {polynomials of order $q^{+}-1$};
    \draw                   (-2,1.75) to ++(3,0);
    \draw[ultra thick,color=black] (-2     ,1.75-0.075) to ++(0,0.15);
    \draw[ultra thick,color=black] (-1.4   ,1.75-0.075) to ++(0,0.15);
    \draw[ultra thick,color=black] (-0.8   ,1.75-0.075) to ++(0,0.15);
    \draw[ultra thick,color=black] (-0.2   ,1.75-0.075) to ++(0,0.15);
    \draw[ultra thick,color=black] ( 0.4   ,1.75-0.075) to ++(0,0.15);
    \draw[ultra thick,color=black] ( 1     ,1.75-0.075) to ++(0,0.15);

    \node[anchor=west] at ( 1.1  ,2.0) {SBP grid values of order $q^{+}$};
    \draw[dashed] (-2,2.0) to ++(3,0);

    \node at (-2.0 ,2.0) {\color{blu} $\bullet$};
    \node at (-1.4 ,2.0) {\color{blu} $\bullet$};
    \node at (-0.8 ,2.0) {\color{blu} $\bullet$};
    \node at (-0.2 ,2.0) {\color{blu} $\bullet$};
    \node at ( 0.4 ,2.0) {\color{blu} $\bullet$};
    \node at ( 1.0 ,2.0) {\color{blu} $\bullet$};

  \end{tikzpicture}
  \caption{Example of the alignment of a glue grid with two SBP finite
  difference grids. On the finite difference grids values are stored at the
  nodal locations whereas the glue grids stores polynomials over the indicated
  intervals. The projection operators associated with the SBP finite difference
  grid are defined to go to glue grid that conform to the SBP grid nodes. Since
  both sides of the interface do not have the same space of piecewise
  polynomials, additional projection operators are needed to move between the
  polynomial spaces.}\label{fig:glue}
\end{figure}
The $\mat{H}$-compatible projection operators discussed in \aref{app:projection}
move between an SBP finite difference solution and a given set of piecewise
polynomials of order $q-1$ where $q$ is the accuracy of the finite difference
method; see for example \fref{fig:gluegrid}. To make these operators useful in
practice we need to be able to transition between glue grid spaces.

To understand why this is, consider the situation shown in \fref{fig:glue}. Here
an interface between two SBP operators with different grids and orders
of accuracy is shown. Since the operators constructed in \aref{app:projection}
move between finite difference grid values and a fixed set of piecewise continuous
polynomials (fixed intervals and orders), the glue grids defined for either side
of the interface in \fref{fig:glue} will not conform (i.e., the polynomials
may be of a different order and/or the locations of the interval boundaries may
be different). Thus, additional projection operators are needed to move
between the different polynomial orders and intervals.

The projection operators in this work are constructed in a hierarchical fashion
where at each stage we need to construct projections between two different
piecewise polynomial glue grid spaces where one is the subset of the other,
i.e., $G_{h}^a \subset G_{h}^b$. The projection operators between spaces are
constructed to satisfy an analog of~\eref{eqn:h:compatible:projections}, namely
\begin{equation}
  \mat{M}_a\mat{P}_{g_b2g_a} = \mat{P}_{g_a2g_b}^{T}\mat{M}_b,
\end{equation}
where $\mat{M}_a$ and $\mat{M}_b$ are the respective glue grid mass matrices
(which are each symmetric positive definite given linearly independent basis
functions).
Since $G_{h}^a \subset G_{h}^b$ then $\mat{P}_{g_a2g_b}$ is a basis
transformation operation which can be determined independent of
$\mat{P}_{g_b2g_a}$ and we have
\begin{equation}
  \label{eqn:M:compatible:projections}
  \mat{P}_{g_b2g_a} = \mat{M}_a^{-1} \mat{P}_{g_a2g_b}^{T}\mat{M}_b.
\end{equation}
\begin{lemma}\label{lemma:proj:stack}
  Suppose $\mat{P}_{f2g_{a}}$ and $\mat{P}_{g_{a}2f}$ satisfy
  \eref{eqn:h:compatible:projections} with $\mat{M} = \mat{M}_{a}$.
  Further suppose $\mat{P}_{g_{b}2g_{a}}$ and $\mat{P}_{g_{a}2g_{b}}$
  satisfy \eref{eqn:M:compatible:projections}, then
  \begin{align}
    \mat{P}_{f2g_{b}} &= \mat{P}_{g_{a}2g_{b}}\mat{P}_{f2g_{a}},&
    \mat{P}_{g_{b}2f} &= \mat{P}_{g_{a}2f}\mat{P}_{g_{b}2g_{a}}
  \end{align}
  satisfy \eref{eqn:h:compatible:projections} with $\mat{M} = \mat{M}_{b}$.
\end{lemma}
\begin{proof}
  By direct calculation we have
  \begin{alignat}{2}
    \notag
    \mat{P}_{g_{b}2f}^{T}\mat{H} &=
    \mat{P}_{g_{b}2g_{a}}^{T}\mat{P}_{g_{a}2f}^{T} \mat{H}
    &&=
    \mat{P}_{g_{b}2g_{a}}^{T}\mat{M}_{a}\mat{P}_{f2g_{a}}\\
    &=
    \mat{M}_{b}\mat{P}_{g_{a}2g_{b}}\mat{P}_{f2g_{a}}
    &&=
    \mat{M}_{b}\mat{P}_{f2g_{b}}
  \end{alignat}
\end{proof}

Further, it follows from the nesting of the glue spaces that the accuracy
conditions~\eqref{eqn:accuracy} are also satisfied by these composition of
projection operators.  Due to this and Lemma~\ref{lemma:proj:stack} we may assume
without loss of generality that $\mat{P}_{f2g}$ and $\mat{P}_{g2f}$ project all
the way through to the finest glue space (represented by the middle glue grid in
\fref{fig:glue}), that is the intermediate spaces are not explicitly considered
further in this work.

\section{Acoustic Wave Equation: SBP-SAT Discretization}\label{sec:single}
As a model problem we consider the two-dimensional acoustic wave equation in
first order form:
\begin{align}
  \label{eqn:acoustic}
  \rho \pd{v_{i}}{t} + \pd{p}{x_{i}}& = 0\ (i = 1,2), &
  \pd{p}{t} + \lambda \left(\pd{v_{1}}{x_{1}} + \pd{v_{2}}{x_{2}}\right)& = 0,
\end{align}
where $v_{1}$ and $v_{2}$ are the particle velocities in the $x_{1}$ and $x_{2}$
directions, respectively, and $p$ is the pressure. Here, $\rho$ is the material
density and $\lambda$ is Lam\'{e}'s second parameter where we assume
$\rho,\lambda > 0$.

We are interested in discretizing \eref{eqn:acoustic} on a domain $\Omega$
which is the union of curvilinear, quadrilateral domains (blocks)
$\{\Omega_{e}\}$. To do so, we transform each domain from the physical space
$\Omega_{e}$ to the reference space $\tilde{\Omega} = [-1,1]\times[-1,1]$ via the
coordinate transform $x_{i} = x_{i}(\xi_{1},\xi_{2})$, $i=1,2$, with $x_{i}$
being the coordinates in physical domain and $\xi_{i}$ being the coordinates in
the reference domain; we assume that the inverse transforms
$\xi_{i}(x_{1},x_{2})$ also exist. (Note, for simplification of notation we
suppress the geometrical terms dependence on each domain.) The Jacobian
determinant is
\begin{equation}
  J = \pd{x_{1}}{\xi_{1}} \pd{x_{2}}{\xi_{2}} -
      \pd{x_{2}}{\xi_{1}} \pd{x_{1}}{\xi_{2}},
\end{equation}
which gives rise to the metric relations
\begin{align}
    J \pd{\xi_{1}}{x_{1}}& = \pd{x_{2}}{\xi_{2}},&
    J \pd{\xi_{1}}{x_{2}}& =-\pd{x_{1}}{\xi_{2}},&
    J \pd{\xi_{2}}{x_{2}}& = \pd{x_{1}}{\xi_{1}},&
    J \pd{\xi_{2}}{x_{1}}& =-\pd{x_{2}}{\xi_{1}}.
\end{align}
With these definitions, the acoustic wave equation \eref{eqn:acoustic} can be
written as
\begin{align}
  \label{eqn:acoustic:transformed:v}
  \rho J \pd{v_{i}}{t}
  + \pd{}{\xi_{1}}\left(J\pd{\xi_{1}}{x_{i}} p\right)
  + \pd{}{\xi_{2}}\left(J\pd{\xi_{2}}{x_{i}} p\right) &= 0,\quad i = 1,2,\\
  \label{eqn:acoustic:transformed:p}
  J\pd{p}{t} + \lambda\left(
    J\pd{\xi_{1}}{x_{1}} \pd{v_{1}}{\xi_{1}}
  + J\pd{\xi_{2}}{x_{1}} \pd{v_{1}}{\xi_{2}}
  + J\pd{\xi_{1}}{x_{2}} \pd{v_{2}}{\xi_{1}}
  + J\pd{\xi_{2}}{x_{2}} \pd{v_{2}}{\xi_{2}}
  \right) &= 0.
\end{align}
Notice that we have written the transformed equations in skew-symmetric form
with the velocity equations \eref{eqn:acoustic:transformed:v} written using a
conservative transform and the pressure equation
\eref{eqn:acoustic:transformed:p} using a non-conservative transform. It is
common to do both terms conservatively, but doing this splitting results in a
provably stable scheme.

Before presenting an SBP discretization of the governing equations, we first
introduce a variational form of the equations on each domain. This is done
to highlight the close connection between SBP finite difference methods and DG
methods. To do this we introduce test functions $w_{i}$, $i=1,2$, and $\varphi$
which belong to some appropriately chosen space. Multiplying the velocity
equation \eref{eqn:acoustic:transformed:v} by $w_{i}$, the pressure equation
\eref{eqn:acoustic:transformed:p} by $\varphi$, and integrating over a domain
$\tilde{\Omega}$ gives
\begin{align}
  \notag
  &\int_{\tilde{\Omega}} w_{i}\left[\rho J \pd{v_{i}}{t}
  + \pd{}{\xi_{1}}\left(J\pd{\xi_{1}}{x_{i}} p\right)
  + \pd{}{\xi_{2}}\left(J\pd{\xi_{2}}{x_{i}} p\right)\right]\;dA\\
  \label{eqn:var:consmon}
  &\qquad=
  -\int_{\partial \tilde{\Omega}} w_{i} S_{J} n_{i} \left( p^* - p\right)ds,
  \quad i=1,2,\\
  \notag
  &\int_{\tilde{\Omega}}
  \varphi \left[J\pd{p}{t} + \lambda\left(
    J\pd{\xi_{1}}{x_{1}} \pd{v_{1}}{\xi_{1}}
  + J\pd{\xi_{2}}{x_{1}} \pd{v_{1}}{\xi_{2}}
  + J\pd{\xi_{1}}{x_{2}} \pd{v_{2}}{\xi_{1}}
  + J\pd{\xi_{2}}{x_{2}} \pd{v_{2}}{\xi_{2}}\right)\right]\;dA\\
  \label{eqn:var:hookes}
  &\qquad=
  -\int_{\partial \tilde{\Omega}} \varphi\lambda S_{J}
  \left(v^*-v\right) ds,
\end{align}
where $p^{*}$ and $v^{*}$ are penalty terms (also known as numerical fluxes)
that satisfy the boundary or interface conditions that connect the domains.
As discussed below in the discretization, these values are derived from the
numerical solution along the edge of the block; in the case of interfaces the
numerical solution from both sides of the interface is considered. For most
methods, these values are the same (up to a possible sign change) on either side
of an interface, though in this work we will need to relax this to account for a
possible projection error, that is the fact that $\mat{P}_{g2f}\mat{P}_{f2g}$ is
not an identity operation.
Here, $S_J$ is the surface Jacobian, $n_{1}$ and $n_{2}$ are the components of
the outward pointing unit normal (in the $x_{1}$ and $x_{2}$ directions,
respectively), and $v = n_{1}v_{1}+n_{2}v_{2}$ is the normal component
of velocity. For the reference domain $\tilde{\Omega} =
[-1,1]\times[-1,1]$ (which will be used for the finite difference
discretization) the surface Jacobian terms and outward pointing
normals for the edge defined by $\xi_{i} = \pm 1$ are
\begin{alignat}{3}
  \label{eqn:surf:Jacobian}
  S_{J} =\;& J\sqrt{{\left(\pd{\xi_{i}}{x_{2}}\right)}^2 + {\left(\pd{\xi_{i}}{x_{1}}\right)}^2},\quad&
  n_{1} =\;& \pm\frac{J}{S_{J}} \pd{\xi_{i}}{x_{2}},\quad&
  n_{2} =\;& \pm\frac{J}{S_{J}} \pd{\xi_{i}}{x_{1}}.
\end{alignat}
With the above definition, the boundary integrals can be rewritten as
\begin{align}
  \notag
  \int_{\partial \tilde{\Omega}}&w_{i} S_{J} n_{i} \left( p^* - p\right)ds\\
  \label{eqn:bound:int:vel}
  =\;& \phantom{+}
  \int_{-1}^{1} {\left[w_{i} S_{J} n_{i} \left( p^* - p\right)\right]}_{\xi_{1}=-1}d\xi_{2}
  +
  \int_{-1}^{1} {\left[w_{i} S_{J} n_{i} \left( p^* - p\right)\right]}_{\xi_{1}= 1}d\xi_{2}\\
  \notag
  &+
  \int_{-1}^{1} {\left[w_{i} S_{J} n_{i} \left( p^* - p\right)\right]}_{\xi_{2}=-1}d\xi_{1}
  +
  \int_{-1}^{1} {\left[w_{i} S_{J} n_{i} \left( p^* - p\right)\right]}_{\xi_{2}= 1}d\xi_{1},
\end{align}
\begin{align}
  \notag
  \int_{\partial \tilde{\Omega}}& \varphi\lambda S_{J} \left(v^*-v\right) ds \\
  \label{eqn:bound:int:pres}
  =\;& \phantom{+}
    \int_{-1}^{1} {\left[\varphi\lambda S_{J} \left(v^*-v\right)\right]}_{\xi_{1}=-1}\;d\xi_{2}
   +\int_{-1}^{1} {\left[\varphi\lambda S_{J} \left(v^*-v\right)\right]}_{\xi_{1}= 1}\;d\xi_{2}\\
  \notag
  &+\int_{-1}^{1} {\left[\varphi\lambda S_{J} \left(v^*-v\right)\right]}_{\xi_{2}=-1}\;d\xi_{1}
   +\int_{-1}^{1} {\left[\varphi\lambda S_{J} \left(v^*-v\right)\right]}_{\xi_{2}= 1}\;d\xi_{1}.
\end{align}

Going back to the differential form of the equations, we discretize the
reference domain $\tilde{\Omega}$ with an $(N_{1}+1) \times (N_{2}+1)$ grid of
equally spaced points. The grid spacing in the $\xi_{i}$ dimension is $h_{i} =
2/N_{i}$ for $i = 1,2$. Thus the $(k,l)$ grid point is at $(\xi_{1},\xi_{2}) =
(kh_{1}-1,lh_{2}-1)$ for $k = 0,\dots,N_{1}$ and $l = 0,\dots,N_{2}$.
We define the pressure solution vector on the grid as
\begin{align}
  \vec{p} &=
  \begin{bmatrix}
    p_{00} & p_{01} & \cdots & p_{0N_{2}} & p_{10} & \cdots & p_{N_{1}N_{2}}
  \end{bmatrix}^{T},
\end{align}
where $p_{kl}$ approximates the pressure $p$ at grid point $(k,l)$; the solution
vectors $\vec{v}_{1}$ and $\vec{v}_{2}$ are similarly defined.
An SBP-SAT semi-discretization (discretization only in space) of
\eref{eqn:acoustic:transformed:v}--\eref{eqn:acoustic:transformed:p} on a domain
$\Omega_e$ using an $(N_{1}+1) \times (N_{2}+1)$ grid is then
\begin{align}
  \label{eqn:dis:consmon}
  &\rho \mat{J} \fd{\vec{v}_{i}}{t}
  + \mat{D}_{1} \mat{J}\xix{1}{i} \vec{p}
  + \mat{D}_{2} \mat{J}\xix{2}{i} \vec{p} =
  - \mat{{H}}^{-1}\vec{\mathcal{F}}_{v_{i}}, \quad i=1,2,\\
  \label{eqn:dis:hookes}
  &\vec{J}\pd{\vec{p}}{t} +\lambda\left(
    \mat{J}\xix{1}{1}\mat{D}_{1}\vec{v}_{1}
  + \mat{J}\xix{2}{1}\mat{D}_{2}\vec{v}_{1}
  + \mat{J}\xix{1}{2}\mat{D}_{1}\vec{v}_{2}
  + \mat{J}\xix{2}{2}\mat{D}_{2}\vec{v}_{2}
  \right)
  =
  -\lambda\mat{H}^{-1}\vec{\mathcal{F}}_{p}.
\end{align}
Here, we have defined the matrices
\begin{align}
  \mat{H}     =\;& \mat{H}_{N_{1}}\kron\mat{H}_{N_{2}},&
  \vec{D}_{1} =\;& \mat{D}_{N_{1}}\kron\mat{I}_{N_{2}},&
  \vec{D}_{2} =\;& \mat{I}_{N_{1}}\kron\mat{D}_{N_{2}},&
\end{align}
where $\mat{I}_{N_{i}}$, $\mat{H}_{N_{i}}$, and $\mat{D}_{N_{i}}$ are all
matrices of size $(N_{i}+1) \times (N_{i}+1)$ with $\mat{I}_{N_{i}}$ being the
identity matrix, and $\mat{H}_{N_{i}}$ and $\mat{D}_{N_{i}}$ being the 1-D SBP
finite difference operators in the $\xi_{i}$ direction. The vectors
$\vec{v}_1$, $\vec{v}_2$, and $\vec{p}$ are the unknown velocities and
pressures at the finite difference grid points. The diagonal
matrices $\mat{J}$, $\xix{1}{1}$, $\xix{1}{2}$, $\xix{2}{1}$, and $\xix{2}{2}$
have the respective geometric factors evaluated at the finite difference grid
points along their diagonal. For example, letting  $J_{kl}$ denote the Jacobian
determinant (or its approximation) at grid point $(k,l)$ we define
\begin{align}
  \mat{J} = \diag
  \begin{bmatrix}
    J_{00} & J_{01} & \cdots & J_{0N_{2}} & J_{10} & \cdots & J_{N_{1}N_{2}}
  \end{bmatrix},
\end{align}
where $\diag(\cdot)$ constructs a diagonal matrix from a given vector; the
other diagonal matrices are defined similarly. For simplicity of the discussion
we assume that the material parameters $\rho$ and $\lambda$ are constants
in $\Omega$.

Before stating the specific form for the penalty terms
$\vec{\mathcal{F}}_{v_{i}}$ and $\vec{\mathcal{F}}_{p}$, we note the
similarities between the SBP-SAT discretization and a DG method based on the
integral form. Namely, if $\mat{H}$ is interpreted as an elemental mass matrix
then the left-hand side of the \eref{eqn:dis:consmon}--\eref{eqn:dis:hookes}
correspond to the body integral terms in
\eref{eqn:var:consmon}--\eref{eqn:var:hookes}.
Similarly, the right-hand side terms correspond to the boundary integrals. Thus,
the SBP-SAT discretization and the DG method have the same discrete structure.
Additionally, what are commonly referred to as penalty terms in SBP-SAT finite
difference methods are of the same form as the numerical flux terms in DG methods. This is
important because even though the methods are developed using different
continuous representations (i.e., differential versus variational form), these
similarities will facilitate the stable coupling using the penalty and flux
terms, see \sref{sec:sbp:dg}.

The penalty terms in \eref{eqn:dis:consmon}--\eref{eqn:dis:hookes} are taken to
be of the form
\begin{align}
  \label{eqn:pen:v}
  \vec{\mathcal{F}}_{v_{i}} =\;&
    \left( \mat{e}_{W} \kron \vec{\mathcal{F}}^{W}_{v_{i}} \right)
  + \left( \mat{e}_{E} \kron \vec{\mathcal{F}}^{E}_{v_{i}} \right)
  + \left( \vec{\mathcal{F}}^{S}_{v_{i}} \kron \mat{e}_{S} \right)
  + \left( \vec{\mathcal{F}}^{N}_{v_{i}} \kron \mat{e}_{N} \right),\\
  \label{eqn:pen:p}
  \vec{\mathcal{F}}_{p} =\;&
    \left( \mat{e}_{W} \kron \vec{\mathcal{F}}^{W}_{p} \right)
  + \left( \mat{e}_{E} \kron \vec{\mathcal{F}}^{E}_{p} \right)
  + \left( \vec{\mathcal{F}}^{S}_{p} \kron \mat{e}_{S} \right)
  + \left( \vec{\mathcal{F}}^{N}_{p} \kron \mat{e}_{N} \right).
\end{align}
Here the subscripts and superscript $W$, $E$, $S$, and $N$ are used to denote
which side of the domain the penalty term correspond to. For instance, $W$ and
$E$ correspond to the \emph{west} and \emph{east} sides of the domain with
$\xi_{1} = -1$ and $\xi_{1} = 1$, respectively. Similarly, $S$ and $N$
correspond to the \emph{south} and \emph{north} with $\xi_{2} = -1$ and $\xi_{2} = 1$.
The vectors $\vec{e}_{W}$ and $\vec{e}_{E}$ have length $N_{1}+1$ and are zero
everywhere except the first and last entry, respectively, where they are $1$,
i.e.,
\begin{alignat}{2}
  \vec{e}_{W}
  =\;&
  \begin{bmatrix}
    1 & 0 & \cdots & 0
  \end{bmatrix}^{T}
  \quad
  \text{and}
  \quad&
  \vec{e}_{E}
  =\;&
  \begin{bmatrix}
    0 & \cdots & 0 & 1
  \end{bmatrix}^{T};
\end{alignat}
the vectors $\vec{e}_{S}$ and $\vec{e}_{N}$ are defined analogously.

The vectors $\vec{\mathcal{F}}^{W}_{v_{i}}$ and $\vec{\mathcal{F}}^{W}_{p}$, of
length $N_{2}+1$, are the actual penalty terms (or flux differences) along the west
side. These vectors are taken to have the form
\begin{alignat}{2}
  \label{eqn:pen:west}
  \vec{\mathcal{F}}^{W}_{v_{i}} =\;& \mat{H}_{2}\mat{S}_{JW} \mat{n}^{W}_{i}
  \left( \vec{p}^*_{W} - \vec{p}_{W} \right),\qquad&
                          \vec{\mathcal{F}}^{W}_{p} =\;& \mat{H}_{2}\mat{S}_{JW}
  \lambda
  \left( \vec{v}^*_{W} - \vec{v}_{W} \right).
\end{alignat}
Here $\mat{S}_{JW}$ and $\mat{n}^{W}_{i}$ are $(N_{2}+1) \times (N_{2}+1)$
diagonal matrices with elements corresponding to the surface Jacobian
terms and outward pointing unit normals along the west face; see
\eref{eqn:surf:Jacobian}. The vectors $\vec{p}_{W}$ and $\vec{v}_{W}$, of length
$N_{2}+1$, are the pressure and normal components of velocity ($v = v_{1}
n_{1} + v_{2} n_{2}$), respectively, at grid points along the west face.
Finally, the vectors $\vec{p}^*_{W}$ and $\vec{v}^*_{W}$ will be set based on
the interface and/or boundary conditions for the block. In the case of interface
conditions, these edge values will need to be set in a consistent manner across
the interface. As noted above, these penalty terms correspond to the numerical
fluxes in DG methods, and stability results through a judicious choice of
penalty or flux.

Stability of the semi-discrete discretization, whether a pure multiblock SBP-SAT
discretization or a coupled SBP-DG method, will be based on energy analysis.
To do so we define the energy in a single SBP block as
\begin{align}
  \label{eqn:energy:single}
  E =
  \frac{\rho}{2}\vec{v}_{1}^{T}\mat{J}\mat{H}\vec{v}_{1}
  + \frac{\rho}{2}\vec{v}_{2}^{T}\mat{J}\mat{H}\vec{v}_{2}
  + \frac{1}{2\lambda}\vec{p}^{T}\mat{J}\mat{H}\vec{p},
\end{align}
and define the total energy in the solution as
\begin{align}
  \mathcal{E} = \sum_{blocks} E.
\end{align}
Since the governing equations \eref{eqn:acoustic} are energy conservative with
the free surface boundary condition, it is appropriate to use the following
definition of discrete stability~\cite{GustafssonKreissOliger1996}:
\begin{definition}[Energy Stability]
  The semi-discrete discretization is said to be stable if
  \begin{align}
    \fd{\mathcal{E}}{t} \le 0.
  \end{align}
\end{definition}

For a single SBP block, the energy dissipation rate is characterized by the
following lemma:
\begin{lemma}\label{lemma:disp:single}
  The single SBP block discretization
  \eref{eqn:dis:consmon}--\eref{eqn:dis:hookes} with penalty terms
  \eref{eqn:pen:v}--\eref{eqn:pen:p} of the form of \eref{eqn:pen:west} has the
  energy dissipation rate
  \begin{align}
    \label{eqn:disp:single}
    \fd{E}{t} = & \sum_{K = \{W,E,S,N\}} \mathcal{D}_{K},\\
    \label{eqn:disp:single:edge}
    \mathcal{D}_{K} = &
    - \vec{v}_{K}^{T}\mat{H}_{K}\mat{S}_{JK}\vec{p}^{*}_{K}
    + \vec{v}_{K}^{T}\mat{H}_{K}\mat{S}_{JK}\vec{p}_{K}
    - {\left(\vec{v}^{*}_{K}\right)}^{T}\mat{H}_{K}\mat{S}_{JK}\vec{p}_{K},
  \end{align}
  where $\mathcal{D}_{K}$ is the dissipation rate along edge $K$ of the block
  with $\mat{H}_{W} = \mat{H}_{E} = \mat{H}_{2}$ and
  $\mat{H}_{S} = \mat{H}_{N} = \mat{H}_{1}$.
\end{lemma}
\begin{proof}
  See \aref{app:lemma:disp:single}.
\end{proof}

The implication of the lemma is that the energy dissipation rate for a single
block is the sum of the dissipation rate for each of its edges.  Thus, we can
prove global semi-discrete energy stability by showing that energy is
dissipated across every interface and boundary.

The crux of a stable coupling is then choosing $\vec{p}^{*}$ and $\vec{v}^{*}$
such that when \eref{eqn:disp:single} is summed over all blocks $d\mathcal{E}/dt
\le 0$. Before continuing on to present how the interface terms $\vec{p}^*$ and
$\vec{v}^*$ are formulated when projection operators are used, we consider the
form of the penalty terms take for an exterior boundary and when the interface
is conforming (matching grid and SBP finite difference scheme across interface).

\subsection{Exterior Boundary Treatment}
Since the focus of this work is interface treatment, we only consider the
zero pressure boundary condition $p = 0$. Numerically, through the penalty
terms, this enforcement is done through a linear combination of a central and
upwind biased penalty; here by central we mean a penalty term that leads to no
energy dissipation. If a block edge $K \in \{W,E,S,N\}$ is an outer boundary,
the penalty terms are taken to be of the form
\begin{align}
  \label{eqn:exterior}
  \vec{p}_{K}^* - \vec{p}_{K} = -\vec{p}_{K},
  \qquad
  \vec{v}_{K}^* - \vec{v}_{K} = \alpha \frac{\vec{p}_{K}}{Z},
\end{align}
where $Z = \sqrt{\rho/\lambda} > 0$ is the impedance of the material. Here the
parameter $\alpha \ge 0$ has been introduced with $\alpha = 0$
being the central boundary treatment and $\alpha = 1$ being the fully upwind
boundary treatment. The following lemma assures that the external boundary
treatment is dissipative.

\begin{lemma}
  If edge $K \in \{W,E,S,N\}$ of an SBP block is an exterior boundary with
  penalty terms of the form \eref{eqn:exterior} then the energy dissipation rate
  for the edge is
  \begin{align}
    \label{eqn:disp:single:edge:bnd}
    \mathcal{D}_{K} = -\frac{\alpha}{Z} \vec{p}_{K}^{T} \mat{H}_{K}\mat{S}_{JK}
    \vec{p}_{K},
  \end{align}
  which is non-positive if $\alpha \ge 0$.
\end{lemma}
\begin{proof}
  Solving penalty term \eref{eqn:exterior} for $\vec{p}_{K}^{*}$ and
  $\vec{v}_{K}^{*}$ gives
  \begin{align}
    \vec{p}_{K}^{*} = 0,
    \qquad
    \vec{v}_{K}^{*} = \vec{v}_{K} + \alpha \frac{\vec{p}_{K}}{Z},
  \end{align}
  and \eref{eqn:disp:single:edge:bnd} follows immediately after substituting
  $\vec{p}_{K}^{*}$ and $\vec{v}_{K}^{*}$ into the edge dissipation rate
  \eref{eqn:disp:single:edge}. The non-positiveness of
  \eref{eqn:disp:single:edge:bnd} follows from the fact that $\mat{H}_{K}$ and
  $\mat{S}_{JK}$ are diagonal, positive definite matrices.
\end{proof}

\subsection{Conforming Interface Treatment}
We call a block interface conforming when the grid and the $\mat{H}$-norm
are the same on both sides of the interface; the latter condition typically
implies that the same SBP finite difference method is being used on both sides
of the interface. The interface conditions are continuity of pressure and the
normal component of velocity:
\begin{align}
  p^{+} = p^{-},
  \qquad
  v^{+} = -v^{-}.
\end{align}
Here we have introduced the superscripts~$+$ and~$-$ to denote the two sides of
the interface. Recall that $v$ is the normal component of the velocity and thus
the minus sign in the velocity condition is due to the fact that the normals are
equal and opposite on either side of the interface. For the minus side of the
interface the penalty terms can then be written as a combination of the central
and upwind penalties:
\begin{align}
  \label{eqn:pen:con:v}
  \vec{p}^* - \vec{p} =\;&
  \vec{p}^{*} - \vec{p}^{-} =
  \frac{1}{2}\left(\vec{p}^{+}-\vec{p}^{-}\right)
  +\alpha\frac{Z}{2}\left( \vec{v}^{+}+\vec{v}^{-} \right),\\
  \label{eqn:pen:con:p}
  \vec{v}^* - \vec{v} =\;&
  \vec{v}^* - \vec{v}^{-} =
  -\frac{1}{2}\left(\vec{v}^{+}+\vec{v}^{-}\right)
  -\alpha\frac{1}{2Z} \left(\vec{p}^{+}-\vec{p}^{-}\right),
\end{align}
where stability results when $\alpha \ge 0$, and the central penalty (zero
energy dissipation) is achieved when $\alpha = 0$ and the upwind penalty
with $\alpha = 1$.

\begin{lemma}
  Consider a single, conforming interface between two SBP blocks with penalty
  terms of the form \eref{eqn:pen:con:v}--\eref{eqn:pen:con:p}. Let
  $\mathcal{D}^{-}$ and $\mathcal{D}^{+}$ be the energy dissipation rate along
  each side of the interface, then
  \begin{align}
    \notag
    \mathcal{D}^{-} + \mathcal{D}^{+}
    =\;&
    -\alpha\frac{Z}{2}
    {\left(\vec{v}^{-}+\vec{v}^{+}\right)}^{T}
    \mat{H} \mat{S}_{J}
    \left(\vec{v}^{-}+\vec{v}^{+}\right)\\
    \label{eqn:disp:edge:con}
    &
    -\alpha\frac{1}{2Z}
    {\left(\vec{p}^{-}-\vec{p}^{+}\right)}^{T}
    \mat{H} \mat{S}_{J}
    \left(\vec{p}^{-}-\vec{p}^{+}\right),
  \end{align}
  is non-positive for $\alpha \ge 0$.
\end{lemma}
\begin{proof}
  Solving \eref{eqn:pen:con:v}--\eref{eqn:pen:con:p} for $\vec{p}^{*}$ and
  $\vec{v}^{*}$ on the minus side of the interface gives
  \begin{align}
    \vec{p}^{*} =\;&
    \frac{1}{2}\left(\vec{p}^{+}+\vec{p}^{-}\right)
    +\alpha\frac{Z}{2}\left( \vec{v}^{+}+\vec{v}^{-} \right),\\
    \vec{v}^* =\;&
    \frac{1}{2}\left(\vec{v}^{-}-\vec{v}^{+}\right)
    +\alpha\frac{1}{2Z} \left(\vec{p}^{-}-\vec{p}^{+}\right).
  \end{align}
  Substituting $\vec{p}^{*}$ and $\vec{v}^{*}$ into \eref{eqn:disp:single:edge}
  on the minus side of the interface results in (after some simplification)
  \begin{align}
    \notag
    \mathcal{D}^{-}
    =\;&
    - \frac{1}{2}{\left(\vec{{v}}^{-}\right)}^{T} \mat{H}\mat{S}_J \vec{{p}}^{+}
    -\alpha\frac{Z}{2}{\left(\vec{{v}}^{-}\right)}^{T} \mat{H}\mat{S}_J \left(
    \vec{{v}}^{+}+\vec{{v}}^{-} \right)\\
    \label{eqn:disp:con:min}
    &+ \frac{1}{2}{\left(\vec{{p}}^{-}\right)}^{T} \mat{H}\mat{S}_J \vec{{v}}^{+}
    +\alpha\frac{1}{2Z} {\left(\vec{{p}}^{-}\right)}^{T} \mat{H}\mat{S}_J
    \left(\vec{{p}}^{+}-\vec{{p}}^{-}\right).
  \end{align}
  A similar calculation for the plus side of the interface gives
  \begin{align}
    \notag
    \mathcal{D}^{+}
    =\;&
    -\frac{1}{2}{\left(\vec{{v}}^{+}\right)}^{T} \mat{H}\mat{S}_J \vec{{p}}^{-}
    -\alpha\frac{Z}{2}{\left(\vec{{v}}^{+}\right)}^{T} \mat{H}\mat{S}_J \left(
    \vec{{v}}^{-}+\vec{{v}}^{+} \right)\\
    \label{eqn:disp:con:plus}
    &+ \frac{1}{2}{\left(\vec{{p}}^{+}\right)}^{T} \mat{H}\mat{S}_J \vec{{v}}^{-}
    +\alpha\frac{1}{2Z} {\left(\vec{{p}}^{+}\right)}^{T} \mat{H}\mat{S}_J
    \left(\vec{{p}}^{-}-\vec{{p}}^{+}\right).
  \end{align}
  Edge energy dissipation \eref{eqn:disp:edge:con} then follows since
  $\mat{H}$ and $\mat{S}_{J}$ are diagonal matrices. Similarly, the
  non-positiveness
  of \eref{eqn:disp:edge:con} when $\alpha \ge 0$ follows from
  the diagonal, positive definiteness of $\mat{H}$ and $\mat{S}_{J}$.
\end{proof}

\section{General interface treatment}\label{sec:noncon:con}
Our discussion of more general interfaces begins with the coupling of two
SBP finite difference blocks that conform at the block level (i.e., no hanging
multiblock nodes). Throughout we assume that both blocks have the same
continuous coordinate transform along the interface. For example, consider the
case shown in \fref{fig:illustration} (left panel), where we assume that the
block on the right side of the interface has been transformed with
$x_{1}^{+}(\xi_{1},\xi_{2})$ and $x_{2}^{+}(\xi_{1},\xi_{2})$, and similarly the
block on the left side has been transformed with $x_{1}^{-}(\xi_{1},\xi_{2})$
and $x_{2}^{-}(\xi_{1},\xi_{2})$. With this notation, both blocks see the same
transform along the interface if $x_{1}^{+}(-1,\xi) = x_{1}^{-}(1,\xi)$ and
$x_{2}^{+}(-1,\xi) = x_{2}^{-}(1,\xi)$, where for simplicity we have assumed
that the west face of the right block is connected to the east face of the left
block. The glue grid is then parameterized by the variable $-1 \le \eta \le 1$.
Note that we assume nothing about how many grid points are along this interface,
only that they conform at the continuous level.

The core idea behind the nonconforming interface treatment is that the penalty
terms are computed on a \emph{glue grid} between the two domains. An example glue
grid between two finite difference methods is shown in \fref{fig:glue}. As can
be seen, the glue grid between the two domains is defined so that the grid
points are nested with the glue grid interval boundaries.

To move values between the finite difference grid and the glue grid the
previously defined projection operators are used. Namely the operators
$\mat{P}^{-}_{f2g}$ and $\mat{P}^{+}_{f2g}$ move values from the grid on the minus and plus
sides of the interface to the glue grid and $\mat{P}^{-}_{g2f}$ and
$\mat{P}^{+}_{g2f}$ from the glue grid to the minus and plus side finite
difference grids. We will see that since at the discrete level both sides of the
interface may sample the geometry and metric terms differently, these geometry
differences, specifically the surface Jacobian, must be taken into account in
the projection to ensure discrete stability.  To do this we project
the square root of the surface Jacobians along with
the grid values to the glue grid; since the surface Jacobian matrices
$\mat{S}_{J}^{\pm}$ are positive, diagonal matrices the square root of these
matrices are trivial to compute. Hence, the values that we work with on the glue
grid are
\begin{align}\label{eqn:project:noncon}
  \vec{\bar{p}}^{\pm} =
  \mat{P}_{f2g}^{\pm} {\left( \mat{S}_{J}^{\pm} \right)}^{1/2} \vec{p}^{\pm},
  \qquad
  \vec{\bar{v}}^{\pm} =
  \mat{P}_{f2g}^{\pm} {\left( \mat{S}_{J}^{\pm} \right)}^{1/2} \vec{v}^{\pm};
\end{align}
we note that values on the glue grid are always scaled by square root of the
surface Jacobian. Here, the vectors $\vec{p}^{\pm}$ and $\vec{v}^{\pm}$ refer
only to pressure values and normal component of velocity along the interface of
interest.

With this notation, the penalty terms along a nonconforming interface
are:
\begin{align}
  \label{eqn:pen:noncon:v}
  \vec{p}^{*} - \vec{p}^{-} =\;&
  {\left(\mat{S}_{J}^{-}\right)}^{-1/2} \mat{P}^{-}_{g2f}
  \left(\vec{\bar{p}}^{*} - \vec{\bar{p}}^{-}\right)
  +\frac{1}{2}\left[
    {\left(\mat{S}_{J}^{-}\right)}^{-1/2} \mat{P}^{-}_{g2f}\vec{\bar{p}}^{-}
    - \vec{p}^{-}
  \right],\\
  \label{eqn:pen:noncon:p}
  \vec{v}^{*} - \vec{v}^{-} =\;&
  {\left(\mat{S}_{J}^{-}\right)}^{-1/2} \mat{P}^{-}_{g2f}
  \left(\vec{\bar{v}}^{*} - \vec{\bar{v}}^{-}\right)
  +\frac{1}{2}\left[
    {\left(\mat{S}_{J}^{-}\right)}^{-1/2} \mat{P}^{-}_{g2f}\vec{\bar{v}}^{-}
    - \vec{v}^{-}
  \right],
\end{align}
where $\vec{\bar{p}}^{*} - \vec{\bar{p}}^{-}$ and $\vec{\bar{v}}^{*} -
\vec{\bar{v}}^{-}$ are defined by \eref{eqn:pen:con:v}--\eref{eqn:pen:con:p}
using the values $\vec{\bar{p}}^{\pm}$ and $\vec{\bar{v}}^{\pm}$ for
$\vec{p}^{\pm}$ and $\vec{v}^{\pm}$, respectively. As in the conforming case,
the parameter $\alpha \ge 0$ controls the central versus upwind biasness of the
scheme. The second term on the right-hand-side of \eref{eqn:pen:noncon:v} (and
\eref{eqn:pen:noncon:p}) is a projection error which arises because
$\mat{P}_{g2f}\mat{P}_{f2g}$ is not an identity operation.

An important implication of the penalty terms
\eref{eqn:pen:noncon:v}--\eref{eqn:pen:noncon:p} is that the projection
operations for the two sides are independent of the scheme on either side of
the interface and the underlying representation of the geometry. This later fact
means that the geometry does not need to be built into the
projection operation. Also note that if the interface is conforming,
the conforming penalties \eref{eqn:pen:con:v}--\eref{eqn:pen:con:p} are
equivalent to the nonconforming penalty terms
\eref{eqn:pen:noncon:v}--\eref{eqn:pen:noncon:p} if the projection matrices are
taken to be the identity matrix: $\mat{P}_{f2g}^{\pm} = \mat{P}_{g2f}^{\pm} =
\mat{I}$.

We can now state the first major result of the paper:
\begin{theorem}\label{thm:noncon:stability}
  Consider a single, nonconforming interface between two SBP blocks with
  penalty terms of the form \eref{eqn:pen:noncon:v}--\eref{eqn:pen:noncon:p}.
  Let $\mathcal{D}^{-}$ and $\mathcal{D}^{+}$ be the energy dissipation rate
  along each side of the interface, then
  \begin{align}
    \notag
    \mathcal{D}^{-} + \mathcal{D}^{+}
    =\;&
    -\alpha\frac{Z}{2}
    {\left(\vec{\bar{v}}^{-}+\vec{\bar{v}}^{+}\right)}^{T}
    \mat{M}
    \left(\vec{\bar{v}}^{-}+\vec{\bar{v}}^{+}\right)\\
    \label{eqn:disp:edge:noncon}
    &
    -\alpha\frac{1}{2Z}
    {\left(\vec{\bar{p}}^{-}-\vec{\bar{p}}^{+}\right)}^{T}
    \mat{M}
    \left(\vec{\bar{p}}^{-}-\vec{\bar{p}}^{+}\right),
  \end{align}
  which is non-positive for $\alpha \ge 0$.
\end{theorem}
\begin{proof}
  Solving \eref{eqn:pen:noncon:v}--\eref{eqn:pen:noncon:p} for $\vec{p}^{*}$ and
  $\vec{v}^{*}$ on the minus side of the interface gives
  \begin{align}
    \vec{p}^{*} =\;&
    {\left(\mat{S}_{J}^{-}\right)}^{-1/2} \mat{P}^{-}_{g2f}
    \left(\vec{\bar{p}}^{*} - \vec{\bar{p}}^{-}\right)
    +\frac{1}{2}\left[
      {\left(\mat{S}_{J}^{-}\right)}^{-1/2} \mat{P}^{-}_{g2f}\vec{\bar{p}}^{-}
      + \vec{p}^{-}
    \right],\\
    \vec{v}^{*} =\;&
    {\left(\mat{S}_{J}^{-}\right)}^{-1/2} \mat{P}^{-}_{g2f}
    \left(\vec{\bar{v}}^{*} - \vec{\bar{v}}^{-}\right)
    +\frac{1}{2}\left[
      {\left(\mat{S}_{J}^{-}\right)}^{-1/2} \mat{P}^{-}_{g2f}\vec{\bar{v}}^{-}
      + \vec{v}^{-}
    \right].
  \end{align}
  Substituting $\vec{p}^{*}$ and $\vec{v}^{*}$ into \eref{eqn:disp:single:edge}
  on the minus side of the interface results in (after some simplification)
  \begin{align}
    \notag
    \mathcal{D}^{-} =\;&
    - {\left(\vec{v}^{-}\right)}^{T}\mat{H}^{-}
    {\left(\mat{S}_{J}^{-}\right)}^{1/2} \mat{P}^{-}_{g2f}
    \left(\vec{\bar{p}}^{*} - \vec{\bar{p}}^{-}\right)
    - \frac{1}{2} {\left(\vec{v}^{-}\right)}^{T}\mat{H}^{-}
    {\left(\mat{S}_{J}^{-}\right)}^{1/2} \mat{P}^{-}_{g2f}\vec{\bar{p}}^{-}
    \\
    &
    - {\left(\vec{p}^{-}\right)}^{T}\mat{H}^{-}
    {\left(\mat{S}_{j}^{-}\right)}^{1/2} \mat{P}^{-}_{g2f}
    \left(\vec{\bar{v}}^{*} - \vec{\bar{v}}^{-}\right)
    - \frac{1}{2}{\left(\vec{p}^{-}\right)}^{T}\mat{H}^{-}
    {\left(\mat{S}_{j}^{-}\right)}^{1/2} \mat{P}^{-}_{g2f}\vec{\bar{v}}^{-}.
  \end{align}
  Using property \eref{eqn:h:compatible:projections} of the projection operator
  the energy dissipation on the minus side of the interface is
  \begin{align}
    \notag
    \mathcal{D}^{-}
    =\;&
    - {\left(\vec{v}^{-}\right)}^{T}
    {\left(\mat{S}_{J}^{-}\right)}^{1/2} {\left(\mat{P}^{-}_{f2g}\right)}^{T}\mat{M}
    \left(\vec{\bar{p}}^{*} - \vec{\bar{p}}^{-}\right)
    - \frac{1}{2} {\left(\vec{v}^{-}\right)}^{T}
    {\left(\mat{S}_{J}^{-}\right)}^{1/2}
    {\left(\mat{P}^{-}_{f2g}\right)}^{T}\mat{M}\vec{\bar{p}}^{-}
    \\ \notag
    &
    - {\left(\vec{p}^{-}\right)}^{T}
    {\left(\mat{S}_{j}^{-}\right)}^{1/2} {\left(\mat{P}^{-}_{f2g}\right)}^{T}\mat{M}
    \left(\vec{\bar{v}}^{*} - \vec{\bar{v}}^{-}\right)
    - \frac{1}{2}{\left(\vec{p}^{-}\right)}^{T}
    {\left(\mat{S}_{j}^{-}\right)}^{1/2} {\left(\mat{P}^{-}_{f2g}\right)}^{T}\mat{M}\vec{\bar{v}}^{-}
    \\
    \label{eqn:noncon:minstar}
    =\;&
    - {\left(\vec{\bar{v}}^{-}\right)}^{T} \mat{M} \vec{\bar{p}}^{*}
    + {\left(\vec{\bar{v}}^{-}\right)}^{T} \mat{M} \vec{\bar{p}}^{-}
    - {\left(\vec{\bar{p}}^{-}\right)}^{T} \mat{M} \vec{\bar{v}}^{*}
  \end{align}
  where we have used that $\mat{H}$ and $\mat{S}_{J}$ commute since they are
  diagonal as well as the definitions of $\vec{\bar{p}}^{-}$ and
  $\vec{\bar{v}}^{-}$; see \eref{eqn:project:noncon}. Solving
  \eref{eqn:pen:con:v} and \eref{eqn:pen:con:p}, evaluated with
  $\vec{\bar{p}}^{\pm}$ and $\vec{\bar{v}}^{\pm}$, for $\vec{\bar{p}}^{*}$ and
  $\vec{\bar{v}}^{*}$ and substituting these values into \eref{eqn:noncon:minstar}
  gives (after minor algebraic manipulations)
  \begin{align}
    \notag
    \mathcal{D}^{-}
    =\;&
    - \frac{1}{2}{\left(\vec{\bar{v}}^{-}\right)}^{T} \mat{M} \vec{\bar{p}}^{+}
    -\alpha\frac{Z}{2}{\left(\vec{\bar{v}}^{-}\right)}^{T} \mat{M} \left(
    \vec{\bar{v}}^{+}+\vec{\bar{v}}^{-} \right)\\
    \label{eqn:noncon:minfinal}
    &+ \frac{1}{2}{\left(\vec{\bar{p}}^{-}\right)}^{T} \mat{M} \vec{\bar{v}}^{+}
    +\alpha\frac{1}{2Z} {\left(\vec{\bar{p}}^{-}\right)}^{T} \mat{M}
    \left(\vec{\bar{p}}^{+}-\vec{\bar{p}}^{-}\right).
  \end{align}
  A similar calculation for the plus side of the interface gives
  \begin{align}
    \notag
    \mathcal{D}^{+}
    =\;&
    -\frac{1}{2}{\left(\vec{\bar{v}}^{+}\right)}^{T} \mat{M} \vec{\bar{p}}^{-}
    -\alpha\frac{Z}{2}{\left(\vec{\bar{v}}^{+}\right)}^{T} \mat{M} \left(
    \vec{\bar{v}}^{-}+\vec{\bar{v}}^{+} \right)\\
    \label{eqn:disp:noncon:plus}
    &+ \frac{1}{2}{\left(\vec{\bar{p}}^{+}\right)}^{T} \mat{M} \vec{\bar{v}}^{-}
    +\alpha\frac{1}{2Z} {\left(\vec{\bar{p}}^{+}\right)}^{T} \mat{M}
    \left(\vec{\bar{p}}^{-}-\vec{\bar{p}}^{+}\right).
  \end{align}
  Summing
  \eref{eqn:noncon:minfinal} and \eref{eqn:disp:noncon:plus} then gives
  \eref{eqn:disp:edge:noncon}. Similarly, the non-positiveness of
  \eref{eqn:disp:edge:noncon} when $\alpha \ge 0$ follows from the positive
  definiteness of $\mat{M}$.
\end{proof}

\paragraph{Comparison with Mattsson and Carpenter~\cite{MattssonCarpenter2010}
Interpolation Operators}
As noted above, Mattsson and Carpenter have previously proposed a set of
SBP-compatible operators for coupling conforming (at the block level) SBP finite
difference methods with a fixed refinement ratio~\cite{MattssonCarpenter2010}.
These operators, which Mattsson and Carpenter denote as $\mat{I}_{F2C}$ and
$\mat{I}_{C2F}$ with $F2C$ and $C2F$ denoting fine to coarse and vice versa,
move a solution all the way from one finite difference grid to the next. Thus,
an important difference with the projection operators we employ here is that
there is an intermediate glue grid which allows the projection
operators to be defined independent of the coupling; Mattsson and Carpenter's
operators depend on both the refinement ratio and SBP operator on either side of
the interface. Additionally, Mattsson and Carpenter had to introduce
additional constraints in order to ensure stability when upwind bias
penalties/numerical fluxes are used; see Equation (16) of Mattsson and
Carpenter~\cite{MattssonCarpenter2010}. In their paper, Mattsson and Carpenter
note that they were unable to construct operators which always satisfy these
constraints and for some cases dissipation was introduced to stabilize the
method. In this work the use of the glue grid allows us to overcome these extra
constraints on the operators as well as the need to introduce dissipation on the
interface; note that in the method we propose here there is dissipation on the
interface and it is controlled by the upwind parameter $\alpha$. Finally, the
compound operators $\mat{I}_{F2C} = \mat{P}^{+}_{g2f}\mat{P}_{f2g}^{-}$ and
$\mat{I}_{C2F} = \mat{P}^{-}_{g2f}\mat{P}_{f2g}^{+}$ satisfy the consistency
constraints (15) and the accuracy conditions of Definition 2.4
of Mattsson and Carpenter~\cite{MattssonCarpenter2010} (though it should be
noted that they are numerically distinct).

\subsection{Many-to-many interfaces}\label{sec:noncon:many}
We now move on to the case when several finite difference blocks are coupled
together along a single interface. As will be seen, the treatment for this case
is identical to the one-to-one interface case except that the surface
Jacobians of the blocks along the coupling interface must be scaled to put
them into the glue grid space. An example of the sort of coupling considered is
shown in the center panel of \fref{fig:illustration} where we are interested in
the treatment of the T-intersection (denoted with a thick line); without loss of
generality we assume that the interface occurs in the $\xi_{2}$ direction for
all blocks.

As noted above, we parameterize the glue space with a variable $-1 \le \eta \le
1$. We let $N^{-}$ be then number of blocks along the minus side of the
interface and each block $k$, $1 \le k \le N^{-}$, overlaps the glue interface
over $\beta^{-(k-1)} \le \eta \le \beta^{-(k)}$ with $\beta^{-(0)} = -1$ and
$\beta^{-(N^{-})} = 1$.
We then define the affine interface transform to take each block interface
(which run from $-1 \le \xi \le 1$) to the appropriate portion of the glue
interface:
\begin{align}
  \label{eqn:eta:xi}
  \eta^{-(k)} &= \frac{\beta^{-(k-1)}(1-\xi_{2}^{-(k)}) + \beta^{-(k)}(1+\xi_{2}^{-(k)})}{2},\\
  \label{eqn:xi:eta}
  \xi_{2}^{-(k)} &=
  \frac{(\eta^{-(k)}-\beta^{-(k)})+(\eta^{-(k)}-\beta^{-(k-1)})}{\beta^{-(k)}-\beta^{-(k-1)}},
\end{align}
where $1 \le k \le N^{-}$ denotes which of the blocks along this side of the
interface we are considering. Since these are affine transforms, their effect on
the surface Jacobian (see \eref{eqn:surf:Jacobian}) will be a constant scaling
of
\begin{align}
  \pd{\xi_{2}^{-(k)}}{\eta^{-(k)}} &= \frac{2}{\beta^{-(k)} - \beta^{-(k-1)}} =
  \frac{1}{\Delta^{-(k)}}.
\end{align}
Here $\Delta^{-(k)}$ is the fraction of the interface which intersects block
$k$. As similar construction is used for the $N^{+}$ blocks on the other side of
the interface.

Projections to the glue and the penalty terms for each block are then defined as
in \eref{eqn:project:noncon}, \eref{eqn:pen:noncon:v}, and
\eref{eqn:pen:noncon:p} except with the surface Jacobians scaled by
$1/\Delta^{\pm(k)}$:
\begin{align}\label{eqn:project:noncon:scaled}
  \vec{\bar{p}}^{\pm(k)} =\;&
  \mat{P}_{f2g}^{\pm(k)} {\left( \frac{\mat{S}_{J}^{\pm(k)}}{\Delta^{\pm(k)}}\right)}^{1/2} \vec{p}^{\pm(k)},
  \qquad
  \vec{\bar{v}}^{\pm(k)} =
  \mat{P}_{f2g}^{\pm(k)} {\left( \frac{\mat{S}_{J}^{\pm(k)}}{\Delta^{\pm(k)}}\right)}^{1/2} \vec{v}^{\pm(k)},\\
  \notag
  \vec{p}^{*(k)} - \vec{p}^{-(k)} =\;&
  {\left(\frac{\mat{S}_{J}^{-(k)}}{\Delta^{-(k)}}\right)}^{-1/2} \mat{P}^{-(k)}_{g2f}
  \left(\vec{\bar{p}}^{*(k)} - \vec{\bar{p}}^{-(k)}\right)\\
  \label{eqn:pen:noncon:v:scaled}
  &+\frac{1}{2}\left[
    {\left(\frac{\mat{S}_{J}^{-(k)}}{\Delta^{-(k)}}\right)}^{-1/2} \mat{P}^{-(k)}_{g2f}\vec{\bar{p}}^{-(k)} - \vec{p}^{-(k)}
  \right],\\
  \notag
  \vec{v}^{*(k)} - \vec{v}^{-(k)} =\;&
  {\left(\frac{\mat{S}_{J}^{-(k)}}{\Delta^{-(k)}}\right)}^{-1/2} \mat{P}^{-(k)}_{g2f}
  \left(\vec{\bar{v}}^{*(k)} - \vec{\bar{v}}^{-(k)}\right)\\
  \label{eqn:pen:noncon:p:scaled}
  &+\frac{1}{2}\left[
    {\left(\frac{\mat{S}_{J}^{-(k)}}{\Delta^{-(k)}}\right)}^{-1/2} \mat{P}^{-(k)}_{g2f}\vec{\bar{v}}^{-(k)} - \vec{v}^{-(k)}
  \right].
\end{align}
Before going on to state the edge dissipation rates, we note that the change of
variables \eref{eqn:eta:xi}--\eref{eqn:xi:eta} requires a slight modification to
the H-compatible definition \eref{def:SBPprojection}, namely we now use the
definition
\begin{align}
  \Delta^{-(k)}{\left(\mat{P}^{-(k)}_{g2f}\right)}^{T}\mat{H}^{-(k)} = \mat{M}^{-(k)}\mat{P}_{f2g}^{-(k)}.
\end{align}
Here, $\mat{H}^{-(k)}$ is the 1-D SBP norm matrix for a grid from $-1$ to $1$ and
$\mat{M}^{-(k)}$ is the mass matrix for the portion of the glue grid running
from $\beta^{-(k-1)}$ to $\beta^{-(k)}$. The difference between the domains of
$\mat{H}^{-(k)}$ and $\mat{M}^{-(k)}$ gives rise to the $\Delta^{-(k)}$ in the
compatibility condition.

We define the solutions vectors on the glue as
\begin{align}
  \vec{\bar{v}}^{\pm} &=
  \begin{bmatrix}
    \vec{\bar{v}}^{\pm(1)}\\
    \vdots\\
    \vec{\bar{v}}^{\pm(N^{\pm})}\\
  \end{bmatrix},
  &
  \vec{\bar{p}}^{\pm} &=
  \begin{bmatrix}
    \vec{\bar{p}}^{\pm(1)}\\
    \vdots\\
    \vec{\bar{p}}^{\pm(N^{\pm})}\\
  \end{bmatrix}
\end{align}
and the block diagonal glue mass matrix
\begin{align}
  \mat{M}
  =
  \begin{bmatrix}
    \mat{M}^{-(1)}\\
    &\ddots\\
    &&\mat{M}^{-(N^{-})}
  \end{bmatrix}
  =
  \begin{bmatrix}
    \mat{M}^{+(1)}\\
    &\ddots\\
    &&\mat{M}^{+(N^{+})}
  \end{bmatrix},
\end{align}
where we note that it is equivalent to define $\mat{M}$ from either the plus or
minus side block mass matrices since they integrate the same space of function
after stacking. With these definitions, we now have the following theorem which
guarantees stability of the interface treatment.

\begin{theorem}\label{thm:many:stability}
  Consider a single, nonconforming interface with $N^{-}$ and $N^{+}$ SBP
  blocks on either side of the interface with
  penalty terms of the form
  \eref{eqn:pen:noncon:v:scaled}--\eref{eqn:pen:noncon:p:scaled}.
  Let $\mathcal{D}^{-}$ and $\mathcal{D}^{+}$ be the energy dissipation rate
  along each side of the interface, then
  \begin{align}
    \notag
    \mathcal{D}^{-} + \mathcal{D}^{+}
    =\;&
    -\alpha\frac{Z}{2}
    {\left(\vec{\bar{v}}^{-}+\vec{\bar{v}}^{+}\right)}^{T}
    \mat{M}
    \left(\vec{\bar{v}}^{-}+\vec{\bar{v}}^{+}\right)\\
    &
    -\alpha\frac{1}{2Z}
    {\left(\vec{\bar{p}}^{-}-\vec{\bar{p}}^{+}\right)}^{T}
    \mat{M}
    \left(\vec{\bar{p}}^{-}-\vec{\bar{p}}^{+}\right),
  \end{align}
  is the non-positive for $\alpha \ge 0$.
\end{theorem}
\begin{proof}
  Solving the penalty terms for
  $\vec{p}^{*(k)}$ and $\vec{v}^{*(k)}$ and substituting into
  \eref{eqn:disp:single:edge} gives (after a calculation similar to
  \eref{eqn:noncon:minstar}) the edge dissipation rate
  \begin{align}
    \notag
    \mathcal{D}^{-(k)}
    =\;&
    - {\left(\vec{\bar{v}}^{-(k)}\right)}^{T}\mat{M}^{-(k)} \vec{\bar{p}}^{*(k)}
    + {\left(\vec{\bar{v}}^{-(k)}\right)}^{T}\mat{M}^{-(k)}
    \vec{\bar{p}}^{-(k)}\\
    \label{eqn:many:sing:disp}
    &- {\left(\vec{\bar{p}}^{-(k)}\right)}^{T}\mat{M}^{-(k)}
    \vec{\bar{v}}^{*(k)}.
  \end{align}
  Defining the vectors
  \begin{align}
    \vec{\bar{v}}^{*} &=
    \begin{bmatrix}
      \vec{\bar{v}}^{*(1)}\\
      \vdots\\
      \vec{\bar{v}}^{*(N^{-})}\\
    \end{bmatrix},
    &
    \vec{\bar{p}}^{*} &=
    \begin{bmatrix}
      \vec{\bar{p}}^{*(1)}\\
      \vdots\\
      \vec{\bar{p}}^{*(N^{-})}\\
    \end{bmatrix},
  \end{align}
  the sum of the contributions from all the blocks on the minus side of the
  interface can be written as
  \begin{align}
    \mathcal{D}^{-} = \sum_{k=1}^{N^{-}} \mathcal{D}^{-(k)}
    =\;&
    - {\left(\vec{\bar{v}}^{-}\right)}^{T}\mat{M}\vec{\bar{p}}^{*}
    + {\left(\vec{\bar{v}}^{-}\right)}^{T}\mat{M}\vec{\bar{p}}
    - {\left(\vec{\bar{p}}^{-}\right)}^{T}\mat{M}\vec{\bar{v}}^{*}.
  \end{align}
  Since this equation is identical to \eref{eqn:noncon:minfinal}, the remainder
  of the proof is identical to that of Theorem~\ref{thm:noncon:stability}.
\end{proof}

\subsection{Connecting with discontinuous Galerkin methods}\label{sec:sbp:dg}
Besides allowing for the stable coupling of general finite difference grids, the
projection operators defined above can also be used to couple finite difference
methods with numerical methods in variational form. To demonstrate this we
consider the coupling of SBP finite difference methods with a curvilinear,
triangle based DG method.  We begin by introducing a
triangular, curved element DG method and then proceed to view each DG element as
a small SBP finite difference block, which leads immediately to a stable
coupling between the methods. Though we use one
particular DG method, coupling with other formulations is possible as the
coupling is purely done at the numerical flux level so any scheme that gives rise to
similar interface terms will be stable. In what follows we only
highlight the essential parts of the DG formulation that are necessary to couple
it with SBP methods and for a more complete description of DG the interested
reader is directed to, for instance, Hesthaven and
Warburton~\cite{HesthavenWarburton2008}.

To introduce the DG method we start with the variational form of the governing
equations \eref{eqn:var:consmon}--\eref{eqn:var:hookes} for a DG element
$\Omega_e$ whose reference element is $\tilde{\Omega}$; for the examples in
this paper we use a triangular reference element. Applying integration by
parts to the conservation of momentum \eref{eqn:var:consmon} in order
to move the spatial derivatives from the solution $p$ to the test functions
$w_{i}$ we get conservation of momentum in the form
\begin{align}
  \label{eqn:var:consmon:weak}
  &\int_{\tilde{\Omega}} \left[ w_{i}\rho J \pd{v_{i}}{t}
  - \pd{w_{i}}{\xi_{1}}J\pd{\xi_{1}}{x_{i}} p
  - \pd{w_{i}}{\xi_{2}}J\pd{\xi_{2}}{x_{i}} p\right]\;dA
  =
  -\int_{\partial \tilde{\Omega}} w_{i} S_{J} n_{i} p^*ds.
\end{align}
Notice that in the right-hand side of \eref{eqn:var:consmon:weak} depends only
on $p^*$, i.e., the value which will become the numerical flux. The form of
\eref{eqn:var:consmon:weak} with the derivative on the test function and
\eref{eqn:var:hookes} with the derivative on the trial function
is sometimes referred to as the skew-symmetric form of the variational
equations.

Discretizing the variational forms \eref{eqn:var:consmon:weak} and
\eref{eqn:var:hookes} in space using the DG method gives rise to the following
semi-discretization on each element:
\begin{align}
  \label{eqn:dg:consmon}
  \rho \mat{M}_{J} \fd{\vec{v}_{i}}{t} =&
  \mat{D}_{1}^{T} \mat{M}_{1i} \vec{p}
  + \mat{D}_{2}^{T} \mat{M}_{2i} \vec{p}
  - \sum_{K=1}^{3} \mat{L}_{K}^{T} \mat{P}_{bc}^{T} \mat{n}_{iK}
  \mat{\Omega}_{bc} \mat{S}_{JK} \vec{p}_{K}^{*},\\
  \notag
  \mat{M}_{J} \fd{\vec{p}}{t} =&
  -\lambda\left(
    \mat{M}_{11} \mat{D}_{1} \vec{v}_{1}
    +
    \mat{M}_{21} \mat{D}_{2} \vec{v}_{1}
    +
    \mat{M}_{12} \mat{D}_{1} \vec{v}_{2}
    +
    \mat{M}_{22} \mat{D}_{2} \vec{v}_{2}
  \right)\\
  \label{eqn:dg:hookes}
  &- \sum_{K=1}^{3}\lambda\mat{L}_{K}^{T} \mat{P}_{bc}^{T} \mat{\Omega}_{bc}
  \mat{S}_{JK} \left(\vec{v}_{K}^* - \vec{v}_{K}^{-}\right),
\end{align}
where the vector $\vec{v}_{K}^{-}$ is the normal component of velocity along
edge $K$ of the element evaluated at the cubature points:
\begin{align}
  \label{eqn:dg:normal:vel}
  \vec{v}^{-}_{K} = \mat{n}_{1K}\mat{P}_{bc}\mat{L}_{K} \vec{v}_{1}
    + \mat{n}_{2K}\mat{P}_{bc}\mat{L}_{K} \vec{v}_{2}.
\end{align}
Here $\mat{L}_{K}$ takes the volume terms to edge $K$ of the element
and $\mat{L}_{K}^{T}$ takes edge $K$ terms to the volume; this is similar
to the behavior of $\mat{e}_{W/E} \kron \mat{I}$ and  $\mat{I} \kron
\mat{e}_{N/S}$ in the SBP method.
Also as in the SBP method, $\mat{D}_{1}$ and $\mat{D}_{2}$ are the reference
element differentiation matrices for the two reference coordinate directions.
Since we will be using curved triangular elements, integration is done using
a cubature in the volume and quadratures along the edges of the elements.
Thus we introduce the projection matrices
$\mat{P}_{c}$ and $\mat{P}_{bc}$ that project from the volume and edge
approximations to the volume and edge cubature points, respectively. At the
cubature locations, the matrices $\mat{\Omega}_{c}$ and $\mat{\Omega}_{bc}$ are
diagonal matrices of the integration weights for the volume and an edge,
respectively.  To ensure stability of the method, we will assume that
$\mat{\Omega}_{c}$ and $\mat{\Omega}_{bc}$ are positive definite. The
element mass matrices in the discretization are defined as
\begin{alignat}{2}
  \mat{M}_{J}  &= \mat{P}_{c}^{T} \mat{\Omega}_{c} \mat{J} \mat{P}_{c},\quad&
  \mat{M}_{ij} &= \mat{P}_{c}^{T} \mat{\Omega}_{c} \mat{J}
                  \mat{\pd{\xi_{i}}{x_{j}}} \mat{P}_{c}.
\end{alignat}
Here the diagonal matrices $\mat{J}$ and $\mat{\pd{\xi_{i}}{x_{j}}}$ are,
respectively, the Jacobian determinant and metric derivatives defined at the
cubature points. The diagonal matrices $\mat{S}_{JK}$ and $\mat{n}_{iK}$
are the surface Jacobian and the components of the unit normal for edge $K$,
respectively, defined at the cubature points.

Defining the energy in a DG element as
\begin{align}\label{eqn:energy:dg}
  E =
  \frac{\rho}{2}\vec{v}_{1}^{T}\mat{M}_{J}\vec{v}_{1}
  + \frac{\rho}{2}\vec{v}_{2}^{T}\mat{M}_{J}\vec{v}_{2}
  + \frac{1}{2\lambda}\vec{p}^{T}\mat{M}_{J}\vec{p}
\end{align}
as well as the edge projected pressures
\begin{align}\label{eqn:dg:edge:pres}
  \vec{p}_{K}^{-} = \mat{P}_{bc}\mat{L}_{K} \vec{p},
\end{align}
the energy dissipation rate for a single DG element can be characterized by the
following lemma.

\begin{lemma}\label{lemma:disp:single:dg}
  The single DG block discretization \eref{eqn:dg:consmon}--\eref{eqn:dg:hookes}
  has the energy dissipation
  rate
  \begin{align}
    \fd{E}{t} = & \sum_{K=1}^{3} \mathcal{D}_{K},\\
    \label{eqn:disp:edge:dg}
    \mathcal{D}_{K} = &
    - {\left(\vec{v}_{K}^{-}\right)}^{T}\mat{\Omega}_{bc}\mat{S}_{JK}\vec{p}_{K}^{*}
    - {\left(\vec{p}_{K}^{-}\right)}^{T}\mat{\Omega}_{bc}\mat{S}_{JK}
    \left(\vec{v}^{*}_{K} - \vec{v}^{-}_{K}\right),
  \end{align}
  with energy as defined in \eref{eqn:energy:dg}
\end{lemma}
\begin{proof}
  Equation~\eqref{eqn:disp:edge:dg} follows directly by inserting
  \eref{eqn:dg:consmon}--\eref{eqn:dg:hookes} into the time derivative of
  \eref{eqn:energy:dg}
  \begin{align}
    \fd{E}{t} = &
    \rho  \vec{v}_{1}^{T}\mat{M}_{J}\fd{\vec{v}_{1}}{t}
    + \rho\vec{v}_{2}^{T}\mat{M}_{J}\fd{\vec{v}_{2}}{t}
    + \frac{1}{\lambda}\vec{p}^{T}\mat{M}_{J}\fd{\vec{p}}{t},
  \end{align}
  and simplifying using the definition of the edge projected pressures
  \eref{eqn:dg:edge:pres} and normal velocity \eref{eqn:dg:normal:vel}.
  The non-positiveness of \eqref{eqn:disp:edge:dg}
  follows from the fact that $\mat{\Omega}_{bc}$ and $\mat{S}_{JK}$ are
  diagonal, positive definite matrices.
\end{proof}

\subsubsection{Boundary and DG-to-DG numerical flux}
When an edge occurs on the physical boundary, the boundary condition $p=0$ is
enforced with the numerical flux
\begin{align}
  \label{eqn:exterior:dg}
  \vec{p}_{K}^* = 0,
  \qquad
  \vec{v}_{K}^* - \vec{v}_{K} = \alpha \frac{\vec{p}_{K}}{Z}.
\end{align}
Similarly, the numerical flux between two DG elements is taken to be
\begin{align}
  \label{eqn:dg:con:v}
  \vec{p}^{*} =\;&
  \frac{1}{2}\left(\vec{p}^{+}+\vec{p}^{-}\right)
  +\alpha\frac{Z}{2}\left( \vec{v}^{+}+\vec{v}^{-} \right),\\
  \label{eqn:dg:con:p}
  \vec{v}^{*} - \vec{v}^{-} =\;&
  -\frac{1}{2}\left(\vec{v}^{+}+\vec{v}^{-}\right)
  -\alpha\frac{1}{2Z} \left(\vec{p}^{+}-\vec{p}^{-}\right).
\end{align}
In both cases, as in the SBP case, the parameter $\alpha$ controls the upwind nature of
the numerical flux. For stability $\alpha \ge 0$ with the central flux resulting if
$\alpha = 0$ and the fully upwind flux if $\alpha = 1$. Note that if
$\vec{p}^{-}$ is subtracted from $\vec{p}^{*}$ these are identical to the
penalty terms previously defined for SBP boundaries \eref{eqn:exterior} and
conforming interfaces \eref{eqn:pen:con:v}--\eref{eqn:pen:con:p}.

\begin{lemma}
  If edge $K$ of a DG element is an exterior boundary with penalty terms of the
  form \eref{eqn:exterior:dg} then the energy dissipation rate for the edge is
  \begin{align}
    \mathcal{D}_{K} = -\frac{\alpha}{Z} {\left(\vec{p}_{K}^{-}\right)}^{T}
    \mat{\Omega}_{bc}\mat{S}_{JK} \vec{p}^{-}_{K},
  \end{align}
  which is non-positive if $\alpha \ge 0$.
\end{lemma}
\begin{proof}
  Follows directly by using \eref{eqn:exterior:dg} in \eref{eqn:disp:edge:dg}.
\end{proof}
\begin{lemma}
  Consider a single interface between two DG elements with penalty terms of the
  form \eref{eqn:dg:con:v}--\eref{eqn:dg:con:p}. Let $\mathcal{D}^{-}$ and
  $\mathcal{D}^{+}$ be the energy dissipation rate along each side of the
  interface, then
  \begin{align}
    \notag
    \mathcal{D}^{-} + \mathcal{D}^{+}
    =\;&
    \label{eqn:disp:edge:dg:interior}
    -\alpha\frac{Z}{2}
    {\left(\vec{v}^{-}+\vec{v}^{+}\right)}^{T}
    \mat{\Omega}_{bc} \mat{S}_{J}
    \left(\vec{v}^{-}+\vec{v}^{+}\right)\\
    &
    -\alpha\frac{1}{2Z}
    {\left(\vec{p}^{-}-\vec{p}^{+}\right)}^{T}
    \mat{\Omega}_{bc} \mat{S}_{J}
    \left(\vec{p}^{-}-\vec{p}^{+}\right),
  \end{align}
  is non-positive for $\alpha \ge 0$.
\end{lemma}
\begin{proof}
  Equation~\eqref{eqn:disp:edge:dg:interior} follows directly by adding
  $\mathcal{D}^{+} + \mathcal{D}^{+}$ and using the definition of the numerical
  flux \eref{eqn:dg:con:v}--\eref{eqn:dg:noncon:p}.
  The non-positiveness of \eqref{eqn:disp:edge:dg:interior}
  follows from the fact that $\mat{\Omega}_{bc}$ and $\mat{S}_{JK}$ are
  diagonal, positive definite matrices.
\end{proof}

\subsubsection{SBP-to-DG interface}
We now consider the case of an edge corresponding to an interface with an SBP
block. Note that in general the edge of an SBP block will be connected to many
DG elements as shown in \fref{fig:illustration} and thus a similar
procedure will be required as was used in \sref{sec:noncon:many} for connecting
many SBP blocks across one interface. For simplicity, we assume that the
DG element only connects to a single SBP block and that the
coupling occurs along the east face of the SBP block (as shown in the right
panel of \fref{fig:illustration}). We index the glue grid using the
SBP coordinate transform, so $\eta = \xi_{2}$ in \fref{fig:illustration} where
$\xi_{2}$ is the second metric coordinate of the SBP block. Let the
DG element intersect the glue grid over the interval
$[\eta_{1},\eta_{2}]$. Since the surface Jacobian for the DG
element is defined for the element's reference space on the boundary of length
$\gamma$, we have to scale the surface Jacobian before projecting to the glue space
as was done in the many-to-one SBP case of \sref{sec:noncon:many}. Thus we
define the scaled and projected DG solution as
\begin{align}
  \vec{\bar{v}}^{-} &= \mat{P}_{f2g}^{-} {\left(\frac{\mat{S}_{J}^{-}}{\Delta^{-}}\right)}^{1/2} \vec{v}^{-},\\
  \vec{\bar{p}}^{-} &= \mat{P}_{f2g}^{-} {\left(\frac{\mat{S}_{J}^{-}}{\Delta^{-}}\right)}^{1/2} \vec{p}^{-},
\end{align}
where $\Delta^{-} = (\eta_2-\eta_1)/\gamma$ and $\mat{P}_{f2g}^{-}$ is the
projection from the DG element edge cubature points to the portion of the glue
grid it overlaps with. Similarly we define the projection back from the glue to
the DG element edge as $\mat{P}_{g2f}^{-}$.

If we use polynomial basis functions of order $q$ for DG and set the glue grid
space to a higher order polynomial space, then we have
that $\mat{P}_{g2f}^{-} \mat{P}_{f2g}^{-} = \mat{I}$, that is there is no
projection error as there was for the SBP solution; note that the converse is
not true as the glue grid is a higher order space. With this assumption, we can
now define the DG numerical flux when connected to the SBP
finite difference solution as
\begin{align}
  \label{eqn:dg:noncon:v}
  \vec{p}^{*} = &\;
  {\left(
      \frac{\mat{S}_{J}^{-}}{\Delta^{-}}
  \right)}^{-1/2}\mat{P}_{g2f}^{-}\bar{\vec{p}}^{*},\\
  \label{eqn:dg:noncon:p}
  \vec{v}^{*} - \vec{v}^{-} =\;&
  {\left(
      \frac{\mat{S}_{J}^{-}}{\Delta^{-}}
  \right)}^{-1/2}
\mat{P}_{g2f}^{-}
  \left(
    \vec{\bar{v}}^{*} - \vec{\bar{v}}^{-}
  \right),
\end{align}
with $\vec{\bar{p}}^{*}$ and $\vec{\bar{v}}^{*} - \vec{\bar{v}}^{-}$ defined as
in \eref{eqn:pen:con:v}--\eref{eqn:pen:con:p} using the values
$\vec{\bar{p}}^{\pm}$ and $\vec{\bar{v}}^{\pm}$ for $\vec{p}^{\pm}$ and
$\vec{v}^{\pm}$, respectively. As in the conforming case,
comparing these numerical flux expressions with the SBP penalty terms
\eref{eqn:pen:noncon:v}--\eref{eqn:pen:noncon:p} we see that they are identical
since there is no projection error going to the glue and back for the
DG solution, in particular since
\begin{align}
  \mat{P}_{g2f}^{-} \mat{P}_{f2g}^{-} \vec{p}^{-} =
  \mat{P}_{g2f}^{-} \vec{\bar{p}}^{-} = \vec{p}^{-}.
\end{align}
Additionally, these numerical fluxes are the same as those in
\eref{eqn:dg:con:v}--\eref{eqn:dg:con:p} since in the case of connecting two DG
elements the projection operators are identity operations, i.e.,
$\mat{P}_{f2g}^{-} = \mat{P}_{g2f}^{-} = \mat{I}$, and the surface Jacobians
are the same for both sides of the interface. With a high enough boundary
cubature order, it can be assumed that
\begin{align}
  \label{eqn:dg:proj}
  \Delta^{-}{\left(\mat{P}^{-}_{g2f}\right)}^{T}\mat{\Omega}_{bc}^{-} = \mat{M}^{-(k)}\mat{P}_{f2g}^{-},
\end{align}
and stability of the SBP-DG coupling is characterized by the following corollary
to Theorem~\ref{thm:many:stability}.

\begin{corollary}
  Consider a single, nonconforming interface between an SBP finite difference
  method and a DG method. If the SBP finite difference
  method has interface penalty terms of the form
  \eref{eqn:pen:noncon:v:scaled}--\eref{eqn:pen:noncon:p:scaled} and the
  DG method has numerical fluxes of the form
  \eref{eqn:dg:noncon:v}--\eref{eqn:dg:noncon:p}, then the coupling interface
  satisfies the dissipation rates of Theorem~\ref{thm:many:stability}.
\end{corollary}
\begin{proof}
  To prove that this corollary is true we will show that a single DG cell,
  indexed by ${}^{(k)}$, satisfies \eref{eqn:many:sing:disp}. The single edge
  dissipation rate for a DG cell comes from substituting
  \eref{eqn:dg:noncon:v}--\eref{eqn:dg:noncon:p} into \eref{eqn:disp:edge:dg}
  and using
  \begin{align}
    \notag
    \mathcal{D}^{-(k)} = &
    - {\left(\vec{v}^{-(k)}\right)}^{T}\mat{\Omega}_{bc}{\left({\Delta^{-(k)}}{\mat{S}_{J}^{-(k)}}\right)}^{1/2}\mat{P}_{g2f}^{-(k)}\vec{\bar{p}}^{*(k)}\\&
    - {\left(\vec{p}^{-(k)}\right)}^{T}\mat{\Omega}_{bc}{\left({\Delta^{-(k)}}{\mat{S}_{J}^{-(k)}}\right)}^{1/2}\mat{P}_{g2f}^{-(k)}\left(\vec{\bar{v}}^{*(k)} - \vec{\bar{v}}^{-(k)}\right).
  \end{align}
  Using \eref{eqn:dg:proj} this can be simplified to \eref{eqn:many:sing:disp}.  Thus,
  the rest of the proof for the dissipation rates follow the same procedure as in
  the proof of Theorem~\ref{thm:many:stability}.
\end{proof}

\section{Numerical Results}\label{sec:results}
Here we confirm the above theoretical stability results as well as explore the
accuracy of the coupling technique.\footnote{MATLAB code for constructing the
interpolation operators used in this section are available in the electronic
supplement and at \url{https://github.com/bfam/sbp_projection_operators}. The
simulation code used to produce the results is available at
\url{https://github.com/bfam/sbp_projection_2d}. For DG, when coupling with
SBP-SAT, we use the code from Hesthaven and
Warburton~\cite{HesthavenWarburton2008} available at
\url{https://github.com/tcew/nodal-dg}.}  A method-of-lines approach
is used to discretize the acoustic wave equation where the spatial
schemes is as described in this paper and an explicit 4th order
Runge--Kutta method is used for the temporal discretization.
The test problem is the discretization of
\eref{eqn:acoustic:transformed:v}--\eref{eqn:acoustic:transformed:p} on the
domain $\Omega = [-1,1]\times[-1,1]$ with $\rho = \lambda = 1$. Zero pressure,
i.e., free surface, boundary conditions are used on all boundaries. We use the
initial condition
\begin{align}
  p(x_{1},x_{2},0) &=
  \cos\left(k_{1} x_{1}\right)\cos\left(k_{1} x_{2}\right)
  +
  \sin\left(k_{2} x_{1}\right)\sin\left(k_{2} x_{2}\right),\\
  v_{i}(x_{1},x_{2},0) &= 0,~ \quad i=1,2,
\end{align}
where $k_{1} = \pi/2$ and $k_{2} = \pi$. With this the exact solution is
\begin{align}
  p(x_{1},x_{2},t) &=
  \cos\left(\omega_{1} t\right)
  \cos\left(k_{1} x_{1}\right)\cos\left(k_{1} x_{2}\right)
  +
  \cos\left(\omega_{2} t\right)
  \sin\left(k_{2} x_{1}\right)\sin\left(k_{2} x_{2}\right),\\
  v_{1}(x_{1},x_{2},t) &=
  \frac{k_{1}}{\omega_{1}} \sin\left(\omega_{1} t\right)
  \sin\left(k_{1} x_{1}\right)\cos\left(k_{1} x_{2}\right)
  -
  \frac{k_{2}}{\omega_{2}} \sin\left(\omega_{2} t\right)
  \cos\left(k_{2} x_{1}\right)\sin\left(k_{2} x_{2}\right),\\
  v_{2}(x_{1},x_{2},t) &=
  \frac{k_{1}}{\omega_{1}} \sin\left(\omega_{1} t\right)
  \cos\left(k_{1} x_{1}\right)\sin\left(k_{1} x_{2}\right)
  -
  \frac{k_{2}}{\omega_{2}} \sin\left(\omega_{2} t\right)
  \sin\left(k_{2} x_{1}\right)\cos\left(k_{2} x_{2}\right),
\end{align}
where $\omega_{j} = k_{j}\sqrt{2}$ for $j=1,2$.

\subsection{One-to-One SBP Coupling}
We first test the coupling of two SBP blocks as shown in the left panel of
\fref{fig:illustration}. The coordinate transforms for the two blocks are
\begin{align}
  \label{eqn:left:coord}
  x_{1}^{(L)}(\xi_{1},\xi_{2}) &=
  \left(\frac{1+\xi_{1}}{10}\right)
  \sin\left(\pi\left(\xi_{2}+1\right)\right)
  - \left(\frac{1-\xi_{1}}{2}\right), &
  x_{2}^{(L)}(\xi_{1},\xi_{2}) &= \xi_{2},\\
  \label{eqn:right:coord}
  x_{1}^{(R)}(\xi_{1},\xi_{2}) &=
  \left(\frac{1-\xi_{1}}{10}\right)
  \sin\left(\pi\left(\xi_{2}+1\right)\right)
  + \left(\frac{1+\xi_{1}}{2}\right), &
  x_{2}^{(R)}(\xi_{1},\xi_{2}) &= \xi_{2},
\end{align}
where the superscript $(L)$ and $(R)$ corresponds to the left and right block,
respectively. Notice that the coordinate transform along the interface is
conforming $x_{1}^{(L)}(1,\xi_{2}) = x_{1}^{(R)}(-1,\xi_{2}) =
\frac{1}{5}\sin\left(\pi\left(\xi_{2}+1\right)\right)$ and
$x_{2}^{(L)}(1,\xi_{2}) = x_{2}^{(R)}(-1,\xi_{2}) = \xi_{2}$.

We discretize the left block with using an $(N/2+1)\times (N+1)$ grid of where
$N = 2^k$ with $k = 6,7,8,9,10$. For the right block we use an
$(M/2+1)\times(M+1)$ grid of points where $M$ is chosen so the interface if
conforming ($M = N$), nested ($M = 2N$), or unnested ($M = 2N+1$). In the
conforming case no projection operator is used, i.e., this is the traditional
SBP-SAT coupling. We run the simulation using SBP orders $q = 2,3,4,5$, where
here $q$ refers to the boundary order, i.e., the interior finite
difference method is of order $p_{i} = 2q$, the boundary finite difference order
is $p_{b} = q$, and the expected rate of convergence for conforming multiblock
SBP is $q+1$. Note that in this work we will exclusively consider the diagonal
norm SBP operators. The final time of the simulations is $t = 1$. For the
penalties we use $\alpha = 1$, thus the interface and boundaries are fully
upwinded.

\begin{table}
  \centering
  \begin{tabular}{rllll}
    \toprule
    & \multicolumn{1}{c}{$q = 2$}
    & \multicolumn{1}{c}{$q = 3$}
    & \multicolumn{1}{c}{$q = 4$}
    & \multicolumn{1}{c}{$q = 5$}\\
    {$N$}
    & \multicolumn{1}{c}{error (rate)}
    & \multicolumn{1}{c}{error (rate)}
    & \multicolumn{1}{c}{error (rate)}
    & \multicolumn{1}{c}{error (rate)}\\
    \midrule
    &
    \multicolumn{4}{c}{conforming meshes (no projection)}\\
    $  64$ & \color{red} $4.3 \times 10^{ -4}$ $\phantom{(0.0)}$ & \color{blu} $2.1 \times 10^{ -4}$ $\phantom{(0.0)}$ & \color{grn} $4.6 \times 10^{ -5}$ $\phantom{(0.0)}$ & \color{org} $3.7 \times 10^{ -5}$ $\phantom{(0.0)}$\\
    $ 128$ & \color{red} $5.2 \times 10^{ -5}$ $         (3.0) $ & \color{blu} $1.4 \times 10^{ -5}$ $         (3.9) $ & \color{grn} $1.4 \times 10^{ -6}$ $         (5.0) $ & \color{org} $5.7 \times 10^{ -7}$ $         (6.0) $\\
    $ 256$ & \color{red} $6.5 \times 10^{ -6}$ $         (3.0) $ & \color{blu} $9.3 \times 10^{ -7}$ $         (4.0) $ & \color{grn} $4.8 \times 10^{ -8}$ $         (4.9) $ & \color{org} $7.7 \times 10^{ -9}$ $         (6.2) $\\
    $ 512$ & \color{red} $8.0 \times 10^{ -7}$ $         (3.0) $ & \color{blu} $5.9 \times 10^{ -8}$ $         (4.0) $ & \color{grn} $1.9 \times 10^{ -9}$ $         (4.6) $ & \color{org} $1.1 \times 10^{-10}$ $         (6.1) $\\
    $1024$ & \color{red} $1.0 \times 10^{ -7}$ $         (3.0) $ & \color{blu} $3.7 \times 10^{ -9}$ $         (4.0) $ & \color{grn} $7.6 \times 10^{-11}$ $         (4.6) $ & \color{org} $1.7 \times 10^{-12}$ $         (6.0) $\\
    \midrule
    &
    \multicolumn{4}{c}{nested meshes}\\
    $  64$ & \color{red} $5.5 \times 10^{ -4}$ $\phantom{(0.0)}$ & \color{blu} $1.4 \times 10^{ -4}$ $\phantom{(0.0)}$ & \color{grn} $8.3 \times 10^{ -5}$ $\phantom{(0.0)}$ & \color{org} $6.5 \times 10^{ -5}$ $\phantom{(0.0)}$\\
    $ 128$ & \color{red} $7.8 \times 10^{ -5}$ $         (2.8) $ & \color{blu} $7.9 \times 10^{ -6}$ $         (4.2) $ & \color{grn} $6.6 \times 10^{ -6}$ $         (3.7) $ & \color{org} $1.3 \times 10^{ -6}$ $         (5.7) $\\
    $ 256$ & \color{red} $1.1 \times 10^{ -5}$ $         (2.8) $ & \color{blu} $4.8 \times 10^{ -7}$ $         (4.0) $ & \color{grn} $3.1 \times 10^{ -7}$ $         (4.4) $ & \color{org} $1.5 \times 10^{ -8}$ $         (6.4) $\\
    $ 512$ & \color{red} $1.6 \times 10^{ -6}$ $         (2.8) $ & \color{blu} $3.2 \times 10^{ -8}$ $         (3.9) $ & \color{grn} $1.3 \times 10^{ -8}$ $         (4.5) $ & \color{org} $1.2 \times 10^{-10}$ $         (7.0) $\\
    $1024$ & \color{red} $2.2 \times 10^{ -7}$ $         (2.8) $ & \color{blu} $2.3 \times 10^{ -9}$ $         (3.8) $ & \color{grn} $6.4 \times 10^{-10}$ $         (4.4) $ & \color{org} $1.5 \times 10^{-12}$ $         (6.3) $\\
    \midrule
    &
    \multicolumn{4}{c}{unnested meshes}\\
    $  64$ & \color{red} $5.0 \times 10^{ -4}$ $\phantom{(0.0)}$ & \color{blu} $1.4 \times 10^{ -4}$ $\phantom{(0.0)}$ & \color{grn} $8.9 \times 10^{ -5}$ $\phantom{(0.0)}$ & \color{org} $6.2 \times 10^{ -5}$ $\phantom{(0.0)}$\\
    $ 128$ & \color{red} $7.1 \times 10^{ -5}$ $         (2.8) $ & \color{blu} $8.0 \times 10^{ -6}$ $         (4.2) $ & \color{grn} $6.7 \times 10^{ -6}$ $         (3.7) $ & \color{org} $1.2 \times 10^{ -6}$ $         (5.6) $\\
    $ 256$ & \color{red} $1.0 \times 10^{ -5}$ $         (2.8) $ & \color{blu} $4.8 \times 10^{ -7}$ $         (4.1) $ & \color{grn} $3.1 \times 10^{ -7}$ $         (4.4) $ & \color{org} $1.4 \times 10^{ -8}$ $         (6.4) $\\
    $ 512$ & \color{red} $1.5 \times 10^{ -6}$ $         (2.8) $ & \color{blu} $3.2 \times 10^{ -8}$ $         (3.9) $ & \color{grn} $1.3 \times 10^{ -8}$ $         (4.5) $ & \color{org} $1.1 \times 10^{-10}$ $         (7.0) $\\
    $1024$ & \color{red} $2.1 \times 10^{ -7}$ $         (2.8) $ & \color{blu} $2.3 \times 10^{ -9}$ $         (3.8) $ & \color{grn} $6.3 \times 10^{-10}$ $         (4.4) $ & \color{org} $1.5 \times 10^{-12}$ $         (6.3) $\\
    \bottomrule
  \end{tabular}
  \caption{Table of the error and estimated convergence rates for the coupling
  of two SBP blocks (see simulations shown in \fref{fig:sbp:2block:error}).
\label{tab:sbp:2block}}
\end{table}
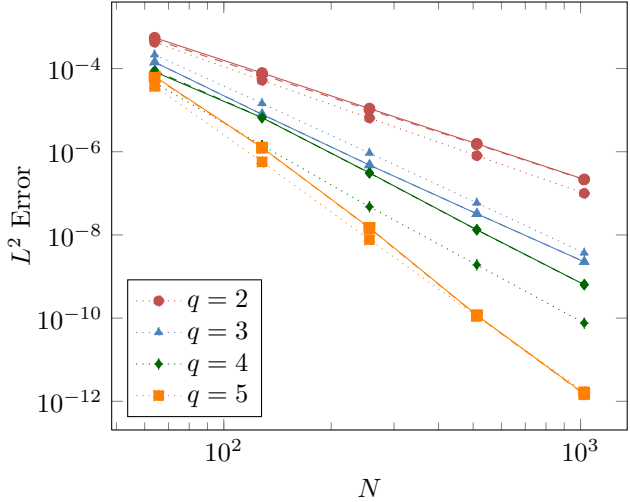
\begin{figure}[tb]
  \centering
  \begin{tikzpicture}
    \begin{loglogaxis}[
        xlabel=\textsc{$N$},
        ylabel=$L^2$ Error,
        legend pos=south west
      ]

      \addplot[color=red,mark=*,dotted] plot coordinates {%
        (64,4.282861e-04)
        (128,5.224703e-05)
        (256,6.458243e-06)
        (512,8.030464e-07)
        (1024,1.001236e-07)
      };
      \addplot[color=blu,mark=triangle*,dotted] plot coordinates {%
        (64,2.143600e-04)
        (128,1.448582e-05)
        (256,9.296184e-07)
        (512,5.868803e-08)
        (1024,3.683632e-09)
      };
      \addplot[color=grn,mark=diamond*,dotted] plot coordinates {%
        (64,4.590937e-05)
        (128,1.421926e-06)
        (256,4.757858e-08)
        (512,1.906394e-09)
        (1024,7.597383e-11)
      };
      \addplot[color=org,mark=square*,dotted] plot coordinates {%
        (64,3.738235e-05)
        (128,5.700696e-07)
        (256,7.710545e-09)
        (512,1.107496e-10)
        (1024,1.693577e-12)
      };

      \addplot[color=red,mark=*] plot coordinates {%
        (64,5.513105e-04)
        (128,7.819699e-05)
        (256,1.093285e-05)
        (512,1.552629e-06)
        (1024,2.163182e-07)
      };
      \addplot[color=blu,mark=triangle*] plot coordinates {%
        (64,1.424577e-04)
        (128,7.913799e-06)
        (256,4.781149e-07)
        (512,3.227392e-08)
        (1024,2.259152e-09)
      };
      \addplot[color=grn,mark=diamond*] plot coordinates {%
        (64,8.354699e-05)
        (128,6.585552e-06)
        (256,3.103095e-07)
        (512,1.330158e-08)
        (1024,6.368893e-10)
      };
      \addplot[color=org,mark=square*] plot coordinates {%
        (64,6.451469e-05)
        (128,1.257878e-06)
        (256,1.488397e-08)
        (512,1.176451e-10)
        (1024,1.514443e-12)
      };

      \addplot[color=red,mark=*,dashed] plot coordinates {%
        (64,4.974287e-04)
        (128,7.083881e-05)
        (256,1.019324e-05)
        (512,1.492158e-06)
        (1024,2.118739e-07)
      };
      \addplot[color=blu,mark=triangle*,dashed] plot coordinates {%
        (64,1.424303e-04)
        (128,7.970779e-06)
        (256,4.793147e-07)
        (512,3.225874e-08)
        (1024,2.256222e-09)
      };
      \addplot[color=grn,mark=diamond*,dashed] plot coordinates {%
        (64,8.945995e-05)
        (128,6.650075e-06)
        (256,3.064335e-07)
        (512,1.318206e-08)
        (1024,6.337922e-10)
      };
      \addplot[color=org,mark=square*,dashed] plot coordinates {%
        (64,6.174023e-05)
        (128,1.243921e-06)
        (256,1.449756e-08)
        (512,1.145746e-10)
        (1024,1.472355e-12)
      };

      \legend{$q=2$\\$q=3$\\$q=4$\\$q=5$\\};
    \end{loglogaxis}
  \end{tikzpicture}
  \caption{Log-log plot of $N$ versus the $L^{2}$ error (measured with the
  energy norm) for the coupling of two SBP blocks (see \fref{fig:illustration}).
  Three different interface types are shown: a conforming interface (dotted
  line), a nested interface with a two-to-one refinement ratio (solid line),
  and an unnested interface with $N_{\text{fine}} = 2N_{\text{coarse}}
  + 1$ (dashed line). In the figure axis $N$ refers to the coarse
  block, i.e., $N_{\text{coarse}} = N$.}\label{fig:sbp:2block:error}
\end{figure}

In \fref{fig:sbp:2block:error} and \tref{tab:sbp:2block} we report the error for
each of the three cases. We measure the error using the $L^{2}$ norm
\begin{align}
  \label{eqn:error:measure}
  \varepsilon^2 =
  \sum_{\{\Omega_e\}}
  \left(
    \frac{\rho}{2}\vec{\Delta v}_{1}^{T}\mat{J}\mat{\bar{H}}\vec{\Delta v}_{1}
    + \frac{\rho}{2}\vec{\Delta v}_{2}^{T}\mat{J}\mat{\bar{H}}\vec{\Delta v}_{2}
    + \frac{1}{2\lambda}\vec{\Delta p}^{T}\mat{J}\mat{\bar{H}}\vec{\Delta p}
  \right),
\end{align}
where $\vec{\Delta v}_{1}$, $\vec{\Delta v}_{2}$, and $\vec{\Delta p}$ are the
difference between the discrete solution and exact solution at all the grid
points; note that this is the same norm used in the stability analysis
\eref{eqn:energy:single}. In  all
the cases the error is decreasing with mesh refinement. For both of the cases
where the projection operators are used, the overall error level is comparable
and slightly higher than the conforming (no projection) case. Also given in
\tref{tab:sbp:2block} are estimates of the convergence rate between two
successive resolutions measured using
\begin{align}
  \label{eqn:error:rate}
  \frac
  {
    \log\left(\varepsilon_{f}\right) - \log\left(\varepsilon_{c}\right)
  }{
    \log\left(N_{c}\right) - \log\left(N_{f}\right)
  },
\end{align}
where here subscript $f$ refers to the finer solution and $c$ the coarser
solution. As the table shows, the coupling does show higher-order convergence
though it is interesting to note that when the projection are used the rates are
more sporadic than the conforming case.

\subsection{Two-to-One SBP Coupling}
We now test the use of the projection operators to couple multiple SBP blocks
along
a single interface. That is we have a single block on the left side of
the interface and two blocks on the right side of the interface as in the center
panel of \fref{fig:illustration}. The block on the left side of the interface
has the coordinate transform \eref{eqn:left:coord}, where as the top and bottom
right blocks have the coordinates transforms:
\begin{align}
  x_{1}^{(RT)}(\xi_{1},\xi_{2}) &=
  x_{1}^{(R)}\left(\frac{\xi_{1}+1}{2},\xi_{2}\right),&
  x_{2}^{(RT)}(\xi_{1},\xi_{2}) &=
  x_{2}^{(R)}\left(\frac{\xi_{1}+1}{2},\xi_{2}\right),\\
  x_{1}^{(RB)}(\xi_{1},\xi_{2}) &=
  x_{1}^{(R)}\left(\frac{\xi_{1}-1}{2},\xi_{2}\right),&
  x_{2}^{(RB)}(\xi_{1},\xi_{2}) &=
  x_{2}^{(R)}\left(\frac{\xi_{1}-1}{2},\xi_{2}\right),
\end{align}
where superscript $(RT)$ and $(RB)$ refer to the right-top and right-bottom
blocks, respectively, and $x_{1,2}^{(R)}$ are defined by \eref{eqn:right:coord}.

We discretize the left block with an $(N/2+1) \times (N+1)$ grid of points, with
$N = 2^k$ for $k=6,7,8,9,10$. The right two blocks are discretized using grids of
size $(N+1) \times (M+1)$ where is chosen so that the interface is nested with a
$1:2$ refinement ratio ($M=N$) or fully unnested ($M = N+1$); in both cases the
interface between the two right blocks is conforming. As before we let
the final time be $t = 1$ and use $\alpha = 1$ in the penalty terms.

\begin{figure}[tb]
  \centering
  \begin{tikzpicture}
    \begin{loglogaxis}[
      xlabel=\textsc{$N$},
      ylabel=$L^2$ Error,
      legend pos=south west
      ]

      \addplot[color=red,mark=*] plot coordinates {%
        (64,5.507031e-04)
        (128,7.817307e-05)
        (256,1.103057e-05)
        (512,1.589351e-06)
        (1024,2.236401e-07)
      };
      \addplot[color=blu,mark=triangle*] plot coordinates {%
        (64,1.421373e-04)
        (128,7.884565e-06)
        (256,4.778790e-07)
        (512,3.230284e-08)
        (1024,2.261244e-09)
      };
      \addplot[color=grn,mark=diamond*] plot coordinates {%
        (64,8.361978e-05)
        (128,6.589251e-06)
        (256,3.107223e-07)
        (512,1.331353e-08)
        (1024,6.362838e-10)
      };
      \addplot[color=org,mark=square*] plot coordinates {%
        (64,6.313077e-05)
        (128,1.217341e-06)
        (256,1.443218e-08)
        (512,1.118591e-10)
        (1024,1.455544e-12)
      };

      \addplot[color=red,mark=*,dashed] plot coordinates {%
        (64,5.008572e-04)
        (128,7.141443e-05)
        (256,1.036944e-05)
        (512,1.538755e-06)
        (1024,2.202376e-07)
      };
      \addplot[color=blu,mark=triangle*,dashed] plot coordinates {%
        (64,1.423896e-04)
        (128,7.942775e-06)
        (256,4.787415e-07)
        (512,3.226398e-08)
        (1024,2.256944e-09)
      };
      \addplot[color=grn,mark=diamond*,dashed] plot coordinates {%
        (64,8.952729e-05)
        (128,6.652057e-06)
        (256,3.066845e-07)
        (512,1.319183e-08)
        (1024,6.329375e-10)
      };
      \addplot[color=org,mark=square*,dashed] plot coordinates {%
        (64,6.192083e-05)
        (128,1.223424e-06)
        (256,1.414144e-08)
        (512,1.093352e-10)
        (1024,1.415718e-12)
      };

      \legend{$q=2$\\$q=3$\\$q=4$\\$q=5$\\};
    \end{loglogaxis}
  \end{tikzpicture}
  \caption{Log-log plot of $N$ versus the $L^{2}$ error (measured with the
  energy norm) for the coupling of three SBP blocks (see \fref{fig:illustration}).
  Two different interface types are shown: a nested interface with a
  two-to-one refinement ratio (solid line) and an unnested interface with
  $N_{\text{fine}} = 2N_{\text{coarse}} + 1$ (dashed line). Also shown
  for reference is the two block conforming interface test (dotted
  line) from \fref{fig:sbp:2block:error}. In the figure axis $N$
  refers to the coarse block, i.e.,
  $N_{\text{coarse}} = N$.}\label{fig:sbp:3block:error}
\end{figure}
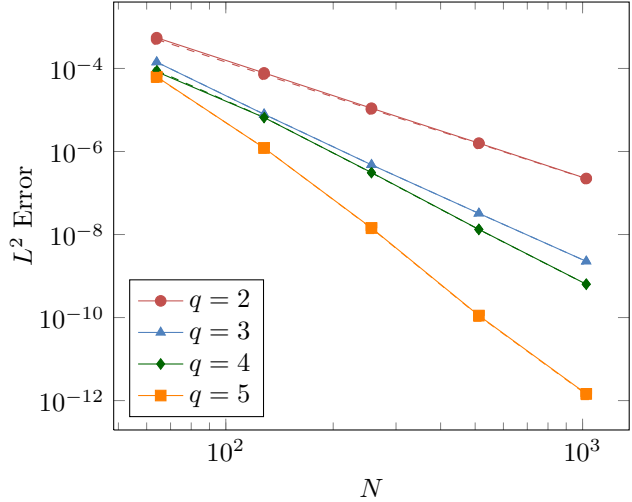
\begin{table}
  \centering
  \begin{tabular}{rllll}
    \toprule
    & \multicolumn{1}{c}{$q = 2$}
    & \multicolumn{1}{c}{$q = 3$}
    & \multicolumn{1}{c}{$q = 4$}
    & \multicolumn{1}{c}{$q = 5$}\\
    {$N$}
    & \multicolumn{1}{c}{error (rate)}
    & \multicolumn{1}{c}{error (rate)}
    & \multicolumn{1}{c}{error (rate)}
    & \multicolumn{1}{c}{error (rate)}\\
    \midrule
    &
    \multicolumn{4}{c}{nested meshes}\\
    $  64$ & \color{red} $5.5 \times 10^{ -4}$ $\phantom{(0.0)}$ & \color{blu} $1.4 \times 10^{ -4}$ $\phantom{(0.0)}$ & \color{grn} $8.4 \times 10^{ -5}$ $\phantom{(0.0)}$ & \color{org} $6.3 \times 10^{ -5}$ $\phantom{(0.0)}$\\
    $ 128$ & \color{red} $7.8 \times 10^{ -5}$ $         (2.8) $ & \color{blu} $7.9 \times 10^{ -6}$ $         (4.2) $ & \color{grn} $6.6 \times 10^{ -6}$ $         (3.7) $ & \color{org} $1.2 \times 10^{ -6}$ $         (5.7) $\\
    $ 256$ & \color{red} $1.1 \times 10^{ -5}$ $         (2.8) $ & \color{blu} $4.8 \times 10^{ -7}$ $         (4.0) $ & \color{grn} $3.1 \times 10^{ -7}$ $         (4.4) $ & \color{org} $1.4 \times 10^{ -8}$ $         (6.4) $\\
    $ 512$ & \color{red} $1.6 \times 10^{ -6}$ $         (2.8) $ & \color{blu} $3.2 \times 10^{ -8}$ $         (3.9) $ & \color{grn} $1.3 \times 10^{ -8}$ $         (4.5) $ & \color{org} $1.1 \times 10^{-10}$ $         (7.0) $\\
    $1024$ & \color{red} $2.2 \times 10^{ -7}$ $         (2.8) $ & \color{blu} $2.3 \times 10^{ -9}$ $         (3.8) $ & \color{grn} $6.4 \times 10^{-10}$ $         (4.4) $ & \color{org} $1.5 \times 10^{-12}$ $         (6.3) $\\
    \midrule
    &
    \multicolumn{4}{c}{unnested meshes}\\
    $  64$ & \color{red} $5.0 \times 10^{ -4}$ $\phantom{(0.0)}$ & \color{blu} $1.4 \times 10^{ -4}$ $\phantom{(0.0)}$ & \color{grn} $8.9 \times 10^{ -5}$ $\phantom{(0.0)}$ & \color{org} $6.2 \times 10^{ -5}$ $\phantom{(0.0)}$\\
    $ 128$ & \color{red} $7.1 \times 10^{ -5}$ $         (2.8) $ & \color{blu} $7.9 \times 10^{ -6}$ $         (4.2) $ & \color{grn} $6.7 \times 10^{ -6}$ $         (3.8) $ & \color{org} $1.2 \times 10^{ -6}$ $         (5.7) $\\
    $ 256$ & \color{red} $1.0 \times 10^{ -5}$ $         (2.8) $ & \color{blu} $4.8 \times 10^{ -7}$ $         (4.1) $ & \color{grn} $3.1 \times 10^{ -7}$ $         (4.4) $ & \color{org} $1.4 \times 10^{ -8}$ $         (6.4) $\\
    $ 512$ & \color{red} $1.5 \times 10^{ -6}$ $         (2.8) $ & \color{blu} $3.2 \times 10^{ -8}$ $         (3.9) $ & \color{grn} $1.3 \times 10^{ -8}$ $         (4.5) $ & \color{org} $1.1 \times 10^{-10}$ $         (7.0) $\\
    $1024$ & \color{red} $2.2 \times 10^{ -7}$ $         (2.8) $ & \color{blu} $2.3 \times 10^{ -9}$ $         (3.8) $ & \color{grn} $6.3 \times 10^{-10}$ $         (4.4) $ & \color{org} $1.4 \times 10^{-12}$ $         (6.3) $\\
    \bottomrule
  \end{tabular}
  \caption{Table of the error and estimated convergence rates for the coupling
  of three SBP blocks (see simulations shown in
  \fref{fig:sbp:3block:error}).}\label{tab:sbp:3block}
\end{table}
In \fref{fig:sbp:3block:error} and \tref{tab:sbp:3block} we report the error (as
measured by \eref{eqn:error:measure}) for SBP orders $q = 2,3,4,5$, where as
before $q$ refers to the boundary order of the SBP method.
In \tref{tab:sbp:3block} we also report the convergence rate as calculated using
\eref{eqn:error:rate}. As these results show, the method maintains the
high-order accuracy of the SBP finite difference when nonconforming block
interfaces are used. As in the case of the one-to-one coupling, the rates are
more sporadic than the conforming one-to-one coupling case.

\subsection{SBP-DG Coupling}\label{sec:sbp:dg:res}
Here we consider the coupling between SBP finite difference methods and DG
finite elements methods as discussed in \sref{sec:sbp:dg}. We specifically use
the curvilinear nodal DG method on triangles as described in Hesthaven and
Warburton~\cite{HesthavenWarburton2008}. The configuration is
as shown in the right panel of \fref{fig:illustration} with an curvilinear SBP
block on the left side of the coupling interface and an unstructured,
curvilinear DG mesh on the right side of the coupling interface. The SBP block
is transformed according to \eref{eqn:left:coord}. For the DG mesh, the element
edges along the curved interface are curved by moving the interpolation nodes to
the interface and the interior interpolation points are then moved using
transfinite blending; elements that are not on the interface remain
straight-sided. Refinement for the unstructured mesh is performed in a
hierarchical fashion with each triangle split into four nested triangles.

As before, we run the SBP mesh with SBP orders $q = 2,3,4,5$ and use a
polynomial order for the DG elements of $q$. Similarly, we discretize the
left SBP block with an $(N/2+1) \times (N+1)$ grid of points, with
$N = 2^k$ for $k=6,7,8,9,10$. The initial DG mesh is chosen so that the number
of unique degrees of freedom along the interface roughly matches the number of
finite difference grid points. That is, we choose a base mesh for each order
that has $\lceil 2^6/(q+1) \rceil$ edges along the coupling interface which is
then refined in a hierarchical manner. The final time for the simulations is $t =
1$ and now consider both $\alpha = 1$ and $\alpha = 0$ for both the SBP
penalties and DG numerical flux terms.

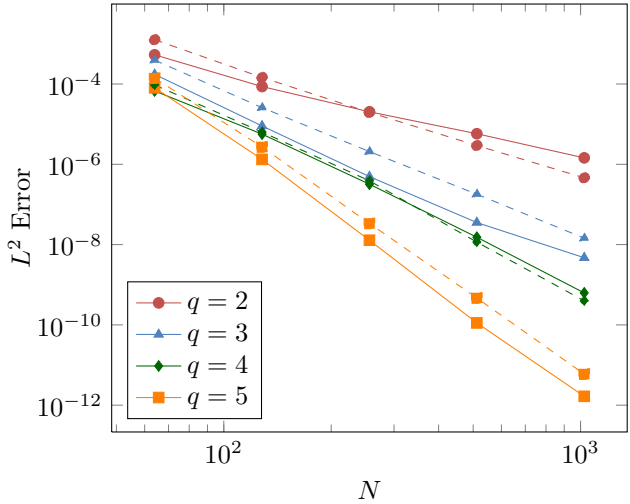
\begin{figure}[tb]
  \centering
  \begin{tikzpicture}
    \begin{loglogaxis}[
      xlabel=\textsc{$N$},
      ylabel=$L^2$ Error,
      legend pos=south west
      ]

      \addplot[color=red,mark=*] plot coordinates {%
        (64,5.339929e-04)
        (128,8.639929e-05)
        (256,2.033917e-05)
        (512,5.772342e-06)
        (1024,1.446614e-06)
      };
      \addplot[color=blu,mark=triangle*] plot coordinates {%
        (64,1.739802e-04)
        (128,9.038869e-06)
        (256,4.934588e-07)
        (512,3.542284e-08)
        (1024,4.705321e-09)
      };
      \addplot[color=grn,mark=diamond*] plot coordinates {%
        (64,6.720900e-05)
        (128,5.650986e-06)
        (256,3.193500e-07)
        (512,1.524211e-08)
        (1024,6.261450e-10)
      };
      \addplot[color=org,mark=square*] plot coordinates {%
        (64,8.056554e-05)
        (128,1.327338e-06)
        (256,1.284980e-08)
        (512,1.118821e-10)
        (1024,1.661564e-12)
      };

      \addplot[color=red,mark=*,dashed] plot coordinates {%
        (64,1.270594e-03)
        (128,1.459187e-04)
        (256,1.967363e-05)
        (512,2.938090e-06)
        (1024,4.604857e-07)
      };
      \addplot[color=blu,mark=triangle*,dashed] plot coordinates {%
        (64,3.935011e-04)
        (128,2.582201e-05)
        (256,2.071674e-06)
        (512,1.797359e-07)
        (1024,1.446322e-08)
      };
      \addplot[color=grn,mark=diamond*,dashed] plot coordinates {%
        (64,9.462354e-05)
        (128,6.363859e-06)
        (256,3.882317e-07)
        (512,1.153309e-08)
        (1024,3.969681e-10)
      };
      \addplot[color=org,mark=square*,dashed] plot coordinates {%
        (64,1.398980e-04)
        (128,2.702084e-06)
        (256,3.364922e-08)
        (512,4.669507e-10)
        (1024,5.950569e-12)
      };

      \legend{$q=2$\\$q=3$\\$q=4$\\$q=5$\\};
    \end{loglogaxis}
  \end{tikzpicture}
  \caption{Log-log plot of $N$ versus the $L^{2}$ error (measured with the
  energy norm) for the coupling of SBP and DG (see \fref{fig:illustration}).
  We plot different lines for runs of different orders where $q$ is both the SBP
  boundary order and the DG polynomial order. The solid line shows the results
  with $\alpha = 1$ and the dashed line results with $\alpha = 0$.
  }\label{fig:sbp:dg:error}
\end{figure}
\begin{table}
  \centering
  \begin{tabular}{rllll}
    \toprule
    & \multicolumn{1}{c}{$q = 2$}
    & \multicolumn{1}{c}{$q = 3$}
    & \multicolumn{1}{c}{$q = 4$}
    & \multicolumn{1}{c}{$q = 5$}\\
    {$N$}
    & \multicolumn{1}{c}{error (rate)}
    & \multicolumn{1}{c}{error (rate)}
    & \multicolumn{1}{c}{error (rate)}
    & \multicolumn{1}{c}{error (rate)}\\
    \midrule
    &
    \multicolumn{4}{c}{$\alpha  = 1$}\\
    $  64$ & \color{red} $5.3 \times 10^{ -4}$ $\phantom{(0.0)}$ & \color{blu} $1.7 \times 10^{ -4}$ $\phantom{(0.0)}$ & \color{grn} $6.7 \times 10^{ -5}$ $\phantom{(0.0)}$ & \color{org} $8.1 \times 10^{ -5}$ $\phantom{(0.0)}$\\
    $ 128$ & \color{red} $8.6 \times 10^{ -5}$ $         (2.6) $ & \color{blu} $9.0 \times 10^{ -6}$ $         (4.3) $ & \color{grn} $5.7 \times 10^{ -6}$ $         (3.6) $ & \color{org} $1.3 \times 10^{ -6}$ $         (5.9) $\\
    $ 256$ & \color{red} $2.0 \times 10^{ -5}$ $         (2.1) $ & \color{blu} $4.9 \times 10^{ -7}$ $         (4.2) $ & \color{grn} $3.2 \times 10^{ -7}$ $         (4.1) $ & \color{org} $1.3 \times 10^{ -8}$ $         (6.7) $\\
    $ 512$ & \color{red} $5.8 \times 10^{ -6}$ $         (1.8) $ & \color{blu} $3.5 \times 10^{ -8}$ $         (3.8) $ & \color{grn} $1.5 \times 10^{ -8}$ $         (4.4) $ & \color{org} $1.1 \times 10^{-10}$ $         (6.8) $\\
    $1024$ & \color{red} $1.4 \times 10^{ -6}$ $         (2.0) $ & \color{blu} $4.7 \times 10^{ -9}$ $         (2.9) $ & \color{grn} $6.3 \times 10^{-10}$ $         (4.6) $ & \color{org} $1.7 \times 10^{-12}$ $         (6.1) $\\
    \midrule
    &
    \multicolumn{4}{c}{$\alpha  = 0$}\\
    $  64$ & \color{red} $1.3 \times 10^{ -3}$ $\phantom{(0.0)}$ & \color{blu} $3.9 \times 10^{ -4}$ $\phantom{(0.0)}$ & \color{grn} $9.5 \times 10^{ -5}$ $\phantom{(0.0)}$ & \color{org} $1.4 \times 10^{ -4}$ $\phantom{(0.0)}$\\
    $ 128$ & \color{red} $1.5 \times 10^{ -4}$ $         (3.1) $ & \color{blu} $2.6 \times 10^{ -5}$ $         (3.9) $ & \color{grn} $6.4 \times 10^{ -6}$ $         (3.9) $ & \color{org} $2.7 \times 10^{ -6}$ $         (5.7) $\\
    $ 256$ & \color{red} $2.0 \times 10^{ -5}$ $         (2.9) $ & \color{blu} $2.1 \times 10^{ -6}$ $         (3.6) $ & \color{grn} $3.9 \times 10^{ -7}$ $         (4.0) $ & \color{org} $3.4 \times 10^{ -8}$ $         (6.3) $\\
    $ 512$ & \color{red} $2.9 \times 10^{ -6}$ $         (2.7) $ & \color{blu} $1.8 \times 10^{ -7}$ $         (3.5) $ & \color{grn} $1.1 \times 10^{ -8}$ $         (5.1) $ & \color{org} $4.7 \times 10^{-10}$ $         (6.2) $\\
    $1024$ & \color{red} $4.6 \times 10^{ -7}$ $         (2.7) $ & \color{blu} $1.4 \times 10^{ -8}$ $         (3.6) $ & \color{grn} $4.0 \times 10^{-10}$ $         (4.9) $ & \color{org} $6.0 \times 10^{-12}$ $         (6.3) $\\
    \bottomrule
  \end{tabular}
  \caption{SBP-DG coupling where $q$ is the SBP boundary order and the DG polynomial
  order}\label{tab:sbp:dg}
\end{table}

Shown in \fref{fig:sbp:dg:error} and \tref{tab:sbp:dg} are the error and
convergence results for this test problem for both values of $\alpha$. As can be
seen, the method is converging at high-order accuracy in all cases. As in the
purely SBP to SBP coupling the convergence rates are rather sporadic.

\paragraph{Eigenvalue Spectrum}
Here we confirm the stability results by looking at the eigenvalue spectrum of
the SBP-DG coupling. To do this we write the fully coupled system as a linear
equation
\begin{align}
  \fd{\vec{u}}{t} = \mat{A}\vec{u},
\end{align}
and then numerically compute the eigenvalues of $\mat{A}$. The energy stability
analysis implies that all the eigenvalues should
have a non-positive real part. Furthermore, when the penalty/numerical flux parameter is
chosen to be $\alpha = 0$ then the eigenvalues should be purely imaginary. To
confirm this, in \fref{fig:sbp:dg:eig} we show the eigenvalue value spectrum for
the of the coupling of the SBP operator with $q = 5$ with the DG using
polynomial order $5$; the mesh configuration used is the first resolution
for this coupling from \sref{sec:sbp:dg:res}. The maximum real part of the
eigenvalue spectrum is $8.16\times10^{-13}$ for the upwind penalty ($\alpha = 1$)
and the maximum magnitude real part of the eigenvalue spectrum is  $2.67\times10^{-13}$
for the central penalty ($\alpha = 0$), thus confirming the theoretical
stability results.

\begin{figure}
  \centering
  \includegraphics{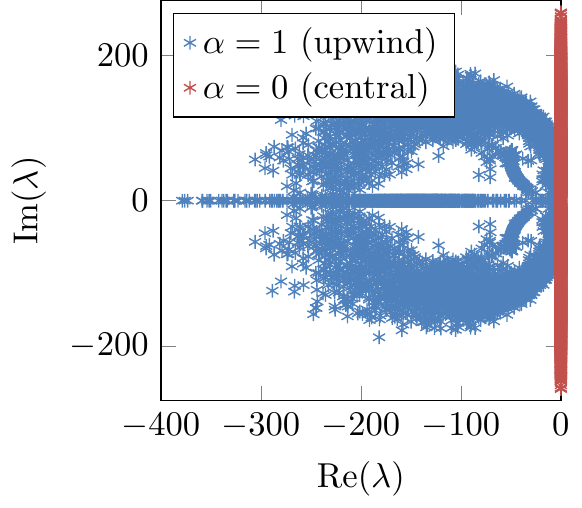}
  \caption{Eigenvalues of the coupled SBP-DG spatial discretization operator for
    the coupling of the SBP operator with $q = 5$ with the DG using polynomial
    order $5$; the mesh configuration used is the first resolution for this
    coupling from \sref{sec:sbp:dg:res}.
    We show results for both upwind $\alpha = 1$ and central $\alpha = 0$
    penalties.
  }\label{fig:sbp:dg:eig}
\end{figure}

\section{Conclusions}\label{sec:conclusions}
In this paper we have presented a new approach to coupling high-order finite
difference methods across nonconforming grid interfaces as well as with
DG methods. The core idea behind the proposed methodology is
the construction of a projection operator that moves the grid solution from the
finite difference points to a finite dimension subspace of the Hilbert space
$L^2(\Gamma)$. The value of this is that once the solution is in this subspace
it can be projected using $L^2$ integral projections to other finite dimensions
subspaces of the Hilbert space $L^2(\Gamma)$ and then projected back to the
finite difference grid. Since the projection operators are consistent with the
SBP H-matrix the fully coupled method is provable stable through the use of weak
enforcement of boundary conditions. In addition to enforcing the boundary
conditions weakly, it was necessary to account for the error in the
projection operator (namely, the fact that the finite
difference grid space and the finite dimensional subspace are not hierarchical
spaces).

In this work we chose a subspace of piecewise polynomial functions which
conformed with the finite difference points. The order of the polynomials used
matched the structure of the SBP finite difference method. That is we required
that the resulting projection operators be exact in the interior for polynomials
of order $q_{i}-1$ and near the boundary for polynomials of order $q_{b} - 1$
with $q_{i}$ and $q_{b}$ being the interior and boundary accuracy of the SBP
finite difference method.

This choice of piecewise polynomial functions as the intermediate space is
tangential to the stability results, and any other finite dimensions subspace of
the Hilbert space $L^2(\Gamma)$ could have been chosen. The reason for this is
that the projection operators we have constructed to move to polynomials can be
used as an intermediate step in moving to any other finite dimensional subspace.
That said, the value of another subspace (including a different set of
piecewise polynomials) would lie in the ability to enforce different accuracy
conditions which may be used to improve the error.

In addition to proposing a new class of SBP-compatible projection operators, we
also showed how these projection operators can be used to account for
differences in the discrete geometry seen by different blocks on either side of
an interface. Namely, we showed that the stability of the numerical method could
be preserved if the geometry terms were projected through with the grid
values. In the work presented here it is assumed that at the continuous level
the coordinate transforms are conforming. The stability results carry over
to the case of dissimilar continuous coordinate transforms, but the approach
outlined in this paper reduces the results to first order accurate. It may be
possible to preserve high-order accuracy in the more general case
where accurate projection operators are constructed between the
coordinate transforms themselves, but this was not explored in this
work.

\section*{Acknowledgments}
We thank the anonymous reviewers of this article for their many helpful
suggestions as well as Lucas Friedrich for his questions and comments concerning
\aref{app:projection}.

\appendix

\section{Proof of Lemma~\ref{lemma:disp:single}}\label{app:lemma:disp:single}
Taking the time derivative of the energy norm \eref{eqn:energy:single} and
substituting in discretization \eref{eqn:dis:consmon}--\eref{eqn:dis:hookes}
gives the energy dissipation rate
\begin{align}
  \fd{E}{t} =\;&
  -\vec{v}_{1}^{T}\mat{Q}_{1} \mat{J}\mat{\pd{\xi_{1}}{x_{1}}} \vec{p}
  -\vec{v}_{1}^{T}\mat{Q}_{2} \mat{J}\mat{\pd{\xi_{2}}{x_{1}}} \vec{p}
  -\vec{v}_{2}^{T}\mat{Q}_{1} \mat{J}\mat{\pd{\xi_{1}}{x_{2}}} \vec{p}
  -\vec{v}_{2}^{T}\mat{Q}_{2} \mat{J}\mat{\pd{\xi_{2}}{x_{2}}} \vec{p}\\\notag
  &
  -\vec{p}^{T}\mat{J}\mat{\pd{\xi_{1}}{x_{1}}}\mat{Q}_{1}\vec{v_{1}}
  -\vec{p}^{T}\mat{J}\mat{\pd{\xi_{2}}{x_{1}}}\mat{Q}_{2}\vec{v_{1}}
  -\vec{p}^{T}\mat{J}\mat{\pd{\xi_{1}}{x_{2}}}\mat{Q}_{1}\vec{v_{2}}
  -\vec{p}^{T}\mat{J}\mat{\pd{\xi_{2}}{x_{2}}}\mat{Q}_{2}\vec{v_{2}}\\\notag
  &
  -\vec{v}_{1}^{T}\vec{\mathcal{F}}_{v_{1}}
  -\vec{v}_{2}^{T}\vec{\mathcal{F}}_{v_{2}}
  -\vec{p}^{T}\vec{\mathcal{F}}_{p},
\end{align}
where
\(
\mat{Q}_{1} = \mat{Q}_{N_1}\kron\mat{H}_{N_2}
\)
and
\(
\mat{Q}_{2} = \mat{H}_{N_1}\kron\mat{Q}_{N_2}.
\)
Using the SBP property $\mat{Q}+\mat{Q}^{T} = \mat{B} =
\diag[-1,0,\dots,0,1]$ and the fact that $\mat{J}$ and
$\mat{\pd{\xi_{j}}{x_{i}}}$ are diagonal matrices (and thus commute), the volume
terms can be transformed to boundary terms:
\begin{align}
  \notag
  \fd{E}{t} =\;&
  -\vec{v}_{1}^{T}\mat{{B}}_{1} \mat{J}\mat{\pd{\xi_{1}}{x_{1}}} \vec{p}
  -\vec{v}_{1}^{T}\mat{{B}}_{2} \mat{J}\mat{\pd{\xi_{2}}{x_{1}}} \vec{p}
  -\vec{v}_{2}^{T}\mat{{B}}_{1} \mat{J}\mat{\pd{\xi_{1}}{x_{2}}} \vec{p}
  -\vec{v}_{2}^{T}\mat{{B}}_{2} \mat{J}\mat{\pd{\xi_{2}}{x_{2}}} \vec{p}\\\notag&
  -\vec{v}_{1}^{T}\vec{\mathcal{F}}_{v_{1}}
  -\vec{v}_{2}^{T}\vec{\mathcal{F}}_{v_{2}}
  -\vec{p}^{T}\vec{\mathcal{F}}_{p}\\
  \label{eqn:energy:rate}
  =\;&
  - \vec{v}_{W}^{T} \mat{H}_{2} \mat{S}_{JW} \vec{p}_{W}
  - \vec{v}_{E}^{T} \mat{H}_{2} \mat{S}_{JE} \vec{p}_{E}
  - \vec{v}_{S}^{T} \mat{H}_{1} \mat{S}_{JS} \vec{p}_{S}
  - \vec{v}_{N}^{T} \mat{H}_{1} \mat{S}_{JN} \vec{p}_{N}\\\notag&
  -\vec{v}_{1}^{T}\vec{\mathcal{F}}_{v_{1}}
  -\vec{v}_{2}^{T}\vec{\mathcal{F}}_{v_{2}}
  -\vec{p}^{T}\vec{\mathcal{F}}_{p},
\end{align}
where \(\mat{{B}}_{1} = \vec{B}_{N_1}\kron\vec{H}_{N_2}\) and
\(\mat{{B}}_{2} = \vec{H}_{N_1}\kron\vec{B}_{N_2}\).
Note that here we have also used the fact that along the block boundaries
\begin{align}
  J \pd{\xi_{i}}{x_{2}} &= \pm n_{1}S_{J},&
  J \pd{\xi_{i}}{x_{1}} &= \pm n_{2}S_{J},
\end{align}
with the positive sign being taken on the ``north'' ($\xi_{2} = 1$) and ``east''
($\xi_{1} = 1$) boundaries and the negative sign being taken on the on the
``south'' ($\xi_{2} = -1$) and ``east'' ($\xi_{1} = -1$) boundaries; see
\eref{eqn:surf:Jacobian}.
Recall also that $\vec{v}_{W}$, $\vec{v}_{E}$, $\vec{v}_{N}$, and $\vec{v}_{S}$
are the normal components of the velocity along the west, east, north, and south
edges; see \eref{eqn:pen:west}.
Using the definition of the penalty terms \eref{eqn:pen:v}--\eref{eqn:pen:p}
along with form \eref{eqn:pen:west}, allows us to rewrite the penalty terms in
\eref{eqn:energy:rate} as
\begin{align}
  \vec{v}_{1}^{T}\vec{\mathcal{F}}_{v_{1}} + \vec{v}_{2}^{T}\vec{\mathcal{F}}_{v_{2}}
  =\;&
  \vec{v}_{W}^{T}\mat{H}_{2}\mat{S}_{JW}\left( \vec{p}^{*}_{W} - \vec{p}_{W} \right)
  + \vec{v}_{E}^{T}\mat{H}_{2}\mat{S}_{JE}\left( \vec{p}^{*}_{E} - \vec{p}_{E} \right)\\
  \notag
  &+ \vec{v}_{S}^{T}\mat{H}_{1}\mat{S}_{JS}\left( \vec{p}^{*}_{S} - \vec{p}_{S} \right)
  + \vec{v}_{N}^{T}\mat{H}_{1}\mat{S}_{JN}\left( \vec{p}^{*}_{N} - \vec{p}_{N} \right)\\
  \vec{p}^{T}\vec{\mathcal{F}}_{p} =\;&
  \vec{p}_{W}^{T}\mat{H}_{2}\mat{S}_{JW}\left( \vec{v}^{*}_{W} - \vec{v}_{W} \right)
  + \vec{p}_{E}^{T}\mat{H}_{2}\mat{S}_{JE}\left( \vec{v}^{*}_{E} - \vec{v}_{E} \right)\\
  \notag
  &+ \vec{p}_{S}^{T}\mat{H}_{1}\mat{S}_{JS}\left( \vec{v}^{*}_{S} - \vec{v}_{S} \right)
  + \vec{p}_{N}^{T}\mat{H}_{1}\mat{S}_{JN}\left( \vec{v}^{*}_{N} - \vec{v}_{N} \right).
\end{align}
These penalty terms can then be used in \eref{eqn:energy:rate} to write
\begin{align}
  \fd{E}{t}
  =
  &
  - \vec{v}_{W}^{T}\mat{H}_{2}\mat{S}_{JW}\vec{p}^{*}_{W}
  + \vec{v}_{W}^{T}\mat{H}_{2}\mat{S}_{JW}\vec{p}_{W}
  - {\left(\vec{v}^{*}_{W}\right)}^{T}\mat{H}_{2}\mat{S}_{JW}\vec{p}_{W}\\
  \notag
  &
  - \vec{v}_{E}^{T}\mat{H}_{2}\mat{S}_{JE}\vec{p}^{*}_{E}
  + \vec{v}_{E}^{T}\mat{H}_{2}\mat{S}_{JE}\vec{p}_{E}
  - {\left(\vec{v}^{*}_{E}\right)}^{T}\mat{H}_{2}\mat{S}_{JE}\vec{p}_{E}\\
  \notag
  &
  - \vec{v}_{S}^{T}\mat{H}_{1}\mat{S}_{JS}\vec{p}^{*}_{S}
  + \vec{v}_{S}^{T}\mat{H}_{1}\mat{S}_{JS}\vec{p}_{S}
  - {\left(\vec{v}^{*}_{S}\right)}^{T}\mat{H}_{1}\mat{S}_{JS}\vec{p}_{S}\\
  \notag
  &
  - \vec{v}_{N}^{T}\mat{H}_{1}\mat{S}_{JN}\vec{p}^{*}_{N}
  + \vec{v}_{N}^{T}\mat{H}_{1}\mat{S}_{JN}\vec{p}_{N}
  - {\left(\vec{v}^{*}_{N}\right)}^{T}\mat{H}_{1}\mat{S}_{JN}\vec{p}_{N},
\end{align}
which is \eref{eqn:disp:single} with the substitution of
\eref{eqn:disp:single:edge}.

\section{Projection Operators}\label{app:projection}
Here we discuss the construction of projection operators that satisfy
Definition~\ref{def:SBPprojection} and the accuracy conditions
\eref{eqn:accuracy}. We will first consider the construction of the operators in
the interior of the domain and then the operators near the boundary. We only
consider diagonal norm SBP operators that have boundary accuracy $p_{b}$ and
interior accuracy $p_{i} = 2p_{b}$. Thus from \eref{eqn:accuracy}, we are 
seeking projection operators $\mat{P}_{f2g}$ and $\mat{P}_{g2f}$ which exactly
project interior polynomials of order $p_{i}-1$ and boundary polynomials of
order $p_{b}-1$.

Let one interval of the glue grid be represented by the basis
${\left\{\phi_{i}\right\}}_{i=0}^{n}$ where $n = p_{i}-1$; we use Legendre
polynomials defined on $[-1, 1]$ in our codes. The piecewise polynomial on the
$k$th glue grid interval (i.e., the glue grid interval $[x_{k},x_{k-1}]$) is
then $p_{n}^{(k)}(x) = \sum_{i=0}^{n} \omega^{(k)}_{i} \phi^{(k)}_{i}(x)$ with
${\left\{\omega^{(k)}_{i}\right\}}_{i=0}^{n}$ being the modal weights which define the function on the
glue. Note that here (and in the following) the superscript $(k)$ on
$\phi_{i}^{(k)}(x)$ is used to signify that polynomial has been shifted to the
interval $[x_{k},x_{k+1}]$, i.e.,
\begin{align}
  \phi_{i}^{(k)}(x) = \phi_{i}\left(\frac{x-x_{k}}{x_{k+1}-x_{k}} -
                                    \frac{x_{k+1}-x}{x_{k+1}-x_{k}}\right).
\end{align}
Note that the basis functions $\psi_{i}(\eta)$ discussed in
\sref{sec:definitions} are basis functions over the whole glue grid, whereas the
basis function $\phi_{i}^{(k)}(x)$ used here are only over a single glue grid
interval.

In the interior, $m=2l$ glue grid intervals are projected to a single grid
point. Thus, if we consider grid solution $q_{k}$ (i.e., the solution at grid
point $k$) we use intervals $k-l$ through $k-1+l$. It is also natural to impose
symmetry conditions on the operator, which then leads to a projection of the
form
\begin{align}
  \label{eqn:app:interior:projection}
  q_{k} = \sum_{j=1}^{l} \sum_{i=0}^{n} \beta_{ij}
  \left[ {(-1)}^{i} \omega_{i}^{(k-j)} + \omega_{i}^{(k-1+j)}\right],
\end{align}
where the ${(-1)}^{i}$ comes from the fact that even modes are symmetric and odd
modes skew-symmetric. Here the coefficient $\beta_{ij}$ is the contribution of
mode $i$ of interval $k+1-j$ and $k+j$ to the grid point value $k$. Note that
$\left\{\beta_{ij}\right\}$ are the values that we are seeking to construct.

For the $s+1$ grid points near the boundary, i.e., grid values $q_{k}$ for
$k=0,\dots,s$, the intervals $1$ through $r$ are used to construct the
projection; by the conditions for stability given in
Definition~\ref{def:SBPprojection} it is required that $s+1 = r+1-l$. The
projection at the boundary then takes the form
\begin{align}
  \label{eqn:app:boundary:projection}
  q_{k} = \sum_{i=0}^{n} \sum_{j=0}^{r-1} \alpha^{(k)}_{ij}
  \omega_{i}^{(j)},~k=0,\dots,s,
\end{align}
where here we allow the boundary polynomials to be of order $n$ even though the
accuracy condition at the boundary are for polynomials of degree $p_{b}-1 =
(n-1)/2$. Hence in total we have $n_{d} = m(n+1) + r(s+1)(n+1)$
coefficients $\left\{\beta_{ij}\right\}$ and
$\left\{\alpha_{ij}^{(k)}\right\}$ to determine.

It is important to remember that by Definition~\ref{def:SBPprojection}, the
structure of the projection from the finite difference grid to the glue is given
by $\mat{P}_{f2g} = \mat{M}^{-1}\mat{P}_{g2f}^{T}\mat{H}$. Thus, once the
structure of $\mat{P}_{g2f}$ is specified $\mat{P}_{f2g}$ is also fixed, i.e.,
no new degrees of freedom are introduced in the problem.

To determine value $\left\{\alpha_{ij}^{(k)}\right\}$ and
$\left\{\beta_{ij}\right\}$, Equations
\eref{eqn:app:interior:projection} and
\eref{eqn:app:boundary:projection} along
with Definitions~\ref{def:SBPprojection} are used as constraints.  To understand
how this is done, consider first \eref{eqn:app:interior:projection}. Let
$\xi(x)$ be a polynomial of degree less than $p_{i}$. Let $\xi(x)$ be
represented on the glue grid by the modal coefficients $\omega_{i}^{(j)}$ for $i
= 0,\dots, n$ and $j = k-l,\dots,k+l-1$.  We want to find values for
$\left\{\beta_{ij}\right\}$ such that $q_{k} = \xi(x_k)$.  Additionally, we want the
$\left\{\beta_{ij}\right\}$ values to have the property that if on the
finite difference grid we have
$q_{j} = \xi(x_{j})$ for $j = k+1-l, \dots, k+l$ then the projection
$\mat{P}_{f2g}$ results in all $n+1$ the modal coefficients being exact for
glue grid interval $k$. The requirement that these two constraints hold for all
polynomials $\xi(x)$ of degree less than $p_i$ results in $p_{i} + p_i(n+1)$ linear
constraints; note that in our code we take $\xi(x)$ to be the Legendre
polynomials.  Constraints for the $\left\{\alpha_{ij}^{(k)}\right\}$
values are derived in an analogous way from
\eref{eqn:app:boundary:projection} except that we require exactness
for polynomials of degree less than $p_{b}$ and thus this introduces
an additional $(s+1)p_{b} + rp_{b}(n+1)$ constraints (the factors $r$ and $s+1$
arise because we have a separate constraint for all $r$ boundary intervals and
$s+1$ boundary grid points). There are an additional $l(n+1)$ constraints due to
the symmetry conditions on $\left\{\beta_{ij}\right\}$ and thus there are a total of
constraints $n_{c} = l(n+1) + p_{i} + p_{i}(n+1) + (s+1)p_{b} + (s+1)p_{b}(n+1)$
to enforce.

\begin{table}
  \centering
  \begin{tabular}{rrrrrrr}
    \toprule
    {$p_{i}$} &
    {$p_{b}$} &
    {$l$} &
    {$s+1$} &
    {$r$} &
    {$n_{c}$} &
    {$n_{d}$}\\
    \midrule
    2  & 1 & 1 &  1 &  1 &   11 &    6\\
    4  & 2 & 2 &  4 &  5 &   76 &   96\\
    6  & 3 & 3 &  6 &  8 &  222 &  324\\
    8  & 4 & 4 &  9 & 12 &  524 &  928\\
    10 & 5 & 5 & 12 & 16 & 1020 & 2020\\
    \bottomrule
  \end{tabular}
  \caption{Parameters used in construction of the SBP compatible projection
   operators. Here $n_{c}$ and $n_{d}$ are the total number of
  constraints and degrees of freedom, respectively.\label{tab:app:parameters}}
\end{table}
We solve these constraint equations numerically using the MATLAB codes included
as an electronic supplement to this paper. The values of $s$ and $l$ needed for
the SBP operators used in this work are given in \tref{tab:app:parameters}
along with the total number of constraints, $n_{c}$, and degrees of freedom
($\left\{\alpha_{ij}^{(k)}\right\}$ and $\left\{\beta_{ij}\right\}$),
$n_{d}$; in the case of $p_{i}=2$ even
though we have more constraints than coefficients a solution does exist and in
all cases there is redundancy in the constraints, that is the linear system does
not have full row rank.  In all but the $p_{i} = 2$ case, the coefficients of the
projection are underdetermined and we use MATLAB's optimization library to
minimize the distance between nearest eigenvalues of $\mat{B} =
\mat{P}_{g2f}\mat{P}_{f2g}$ for a finite difference grid of size $N$. The
motivation for this optimization is to (in some sense) minimize the projection
error by using the degrees of freedom to make $\mat{B}$ closer to an identity
operation.  We use the value of $N=64$, which has been chosen to make the
optimization problem tractable. It is important to note that the optimization is
being used to fix the remaining degrees of freedom after the stability condition
\eref{eqn:h:compatible:projections} is satisfied, i.e., stability does not
depend on the choice of the optimization objective function and other
types of optimization could be considered.

MATLAB routines as well as the final optimized coefficients can be found in the
electronic supplement to this paper as well as in the github repository located
at \url{https://github.com/bfam/sbp_projection_operators}.

\newpage

\bibliographystyle{siam}
\bibliography{long,nps}

\end{document}